\documentclass[a4paper,twoside,english,10pt]{amsart}

\usepackage{fullpage}
\usepackage{amsmath,amsfonts,amssymb}
\usepackage[all]{xy}
\usepackage[english]{babel}
\usepackage[utf8x]{inputenc}
\usepackage{url}
\usepackage{ifthen}
\usepackage{booktabs}
\usepackage{multirow}
\usepackage{array}
\usepackage{paralist}
\usepackage{enumerate}
\usepackage{graphicx}
\usepackage{rotating}

 \newtheorem{thm}{Theorem}[section]
 \newtheorem{prop}[thm]{Proposition}
 \newtheorem{cor}[thm]{Corollary}

\theoremstyle{definition}
 
 \newtheorem{rem}[thm]{Remark}
 \newtheorem{ex}[thm]{Example}
\newtheorem{note}[thm]{Note}

\newcommand{\K}{\mathbb{K}}
\newcommand{\N}{\mathbb{N}}
\newcommand{\Z}{\mathbb{Z}}
\newcommand{\R}{\mathbb{R}}
\newcommand{\C}{\mathbb{C}}

\newcommand{\st}{\;:\;}

\newcommand{\sspace}{{\cdot}}
\newcommand{\ssspace}{{\cdot\cdot}}

\newcommand{\del}{\partial}
\newcommand{\delbar}{\overline{\del}}

\newcommand{\n}{\mathfrak{n}}
\newcommand{\g}{\mathfrak{g}}

\newcommand{\solvmfd}{\left. \Gamma \middle\backslash G \right.}

\newcommand{\cp}{\smile}

\newcommand{\arxiv}[1]{\texttt{#1}}

\newcommand{\paragrafo}[2]{\smallskip \noindent \texttt{Step {#1}} -- {\itshape #2}\\}

\newcommand{\slf}{\mathfrak{sl}}

\DeclareMathOperator{\id}{id}

\DeclareMathOperator{\imm}{im}

\DeclareMathOperator{\de}{d}

\DeclareMathOperator{\esp}{e}

\DeclareMathOperator{\GL}{GL}

\DeclareMathOperator{\End}{End}
\DeclareMathOperator{\Aut}{Aut}
\DeclareMathOperator{\Hom}{Hom}
\DeclareMathOperator{\Der}{Der}

\DeclareMathOperator{\Char}{Char}

\DeclareMathOperator{\Tot}{Tot}
\DeclareMathOperator{\Doub}{Doub}

\DeclareMathOperator{\ad}{ad}
\DeclareMathOperator{\Ad}{Ad}

\DeclareMathOperator{\diag}{diag}

\newcommand{\gl}{\mathfrak{gl}}

\allowdisplaybreaks[1]

\title{Symplectic Bott-Chern cohomology of solvmanifolds}

\author{Daniele Angella}

\address[Daniele Angella]{
Dipartimento di Matematica e Informatica ``Ulisse Dini''\\
Universit\`a degli Studi di Firenze\\
viale Morgagni 67/a\\
50134 Firenze, Italy
}

\email{daniele.angella@gmail.com}
\email{daniele.angella@unifi.it}

\author{Hisashi Kasuya}
\address[Hisashi Kasuya]{Department of Mathematics\\
Tokyo Institute of Technology\\
1-12-1, O-okayama, Meguro\\
Tokyo 152-8551, Japan}
\email{khsc@ms.u-tokyo.ac.jp}
\email{kasuya@math.titech.ac.jp}

\keywords{Dolbeault cohomology, Bott-Chern cohomology, solvmanifolds, invariant symplectic structure}
\thanks{During the preparation of this work, the first author has been granted with a research fellowship by Istituto Nazionale di Alta Matematica INdAM and with a Junior Visiting Position at Centro di Ricerca Matematica ``Ennio de Giorgi'' in Pisa, and he is supported by the Project PRIN ``Varietà reali e complesse: geometria, topologia e analisi armonica'', by the Project FIRB ``Geometria Differenziale e Teoria Geometrica delle Funzioni'', by project SIR 2014 AnHyC ``Analytic aspects in complex and hypercomplex geometry'' (code RBSI14DYEB), and by GNSAGA of INdAM.
The second author  is supported by JSPS Research Fellowships for Young Scientists.}
\subjclass[2010]{22E25; 53C30; 57T15; 53C55; 53D05}

\begin{document}

\begin{abstract}
 We study the symplectic Bott-Chern cohomology by L.-S. Tseng and S.-T. Yau for solvmanifolds endowed with left-invariant symplectic structures.
Our results are applicable to cohomology with values in local systems.
Studying symplectic Bott-Chern cohomology of solvmanifolds with values in local systems,
we give some remarks on symplectic Hodge theory.
\end{abstract}

\maketitle

\section*{Introduction}

In complex geometry, besides Dolbeault and de Rham cohomologies, also Bott-Chern and Aeppli cohomologies deserve much interest. For a complex manifold $X$, the {\em Bott-Chern cohomology} and the {\em Aeppli cohomology} are defined as
$$ H^{\bullet,\bullet}_{BC}(X) \;:=\; \frac{\ker\del\cap\ker\delbar}{\imm\del\delbar} \;, \qquad \text{ respectively } \qquad H^{\bullet,\bullet}_{A}(X) \;:=\; \frac{\ker\del\delbar}{\imm\del+\imm\delbar} \;. $$
In fact, they appear as natural tools in studying the geometry of complex (possibly non-K\"ahler) manifolds, see, e.g., \cite{aeppli, bott-chern, deligne-griffiths-morgan-sullivan, varouchas, alessandrini-bassanelli, schweitzer, kooistra, bismut, tseng-yau-3, angella-1, angella-tomassini-3}.

\medskip

Let $X$ be a $2n$-dimensional manifold endowed with a symplectic structure $\omega$.
We define the co-differential operator $\de^\Lambda := \left[\de,-\iota_{\omega^{-1}}\right]$ on the de Rham complex $\wedge^{\bullet}X$ which was introduced by  J.-L. Koszul, \cite{koszul}, and by J.-L. Brylinski, \cite{brylinski}.
In \cite{tseng-yau-1, tseng-yau-2, tseng-yau-3}, L.-S. Tseng and S.-T. Yau introduced and studied the cohomologies  
$$ H^{\bullet}_{BC}(X;\R) \;:=\; \frac{\ker\de\cap\ker\de^\Lambda}{\imm\de\de^\Lambda} \qquad \text{ and } \qquad H^{\bullet}_{A}(X;\R) \;:=\; \frac{\ker\de\de^\Lambda}{\imm\de+\imm\de^\Lambda} \;. $$
Since they are analogous to Bott-Chern and Aeppli cohomologies of complex manifolds respectively, we call  $ H^{\bullet}_{BC}(X;\R) $  the {\em symplectic Bott-Chern cohomology} and 
$ H^{\bullet}_{A}(X;\R) $ the {\em symplectic Aeppli cohomology}.
We notice that  Tseng and Yau did not use this terminology.
But this  terminology is very useful in this paper for using the result in \cite{angella-kasuya-1}, see Corollary \ref{ives} and its proof.
Moreover, we remark that  $ H^{\bullet}_{BC}(X;\R) $ and $ H^{\bullet}_{A}(X;\R) $  can be considered as the  Bott-Chern and Aeppli cohomologies of a ``generalized'' complex manifold, (see \cite{angella-tomassini-5},) where we regard a symplectic structure as a generalized complex structure of type $0$, (see \cite{cavalcanti, gualtieri}).

We recall that {\em solvmanifolds} are compact quotients $\solvmfd$ of connected simply-connected solvable Lie groups $G$ by co-compact discrete subgroups $\Gamma$.
When $G$ is nilpotent, we call $\solvmfd$ {\em nilmanifold}.
The purpose of this paper is to study the symplectic Bott-Chern and Aeppli cohomologies of symplectic solvmanifolds.
Let $\solvmfd$ be a solvmanifold and $\g$ the Lie algebra of $G$.
Consider the cochain complex $\wedge^{\bullet}\g^{\ast}$ of left-invariant differential forms and the inclusion
$$\iota: \wedge^{\bullet}\g^{\ast}\hookrightarrow \wedge^{\bullet}\solvmfd .
$$
This inclusion induces the map $H^{\bullet}(\g;\R)\to H^{\bullet}_{dR}(\solvmfd;\R)$ where $H^{\bullet}(\g;\R)$ is the cohomology of $\wedge^{\bullet}\g^{\ast}$ and $ H^{\bullet}_{dR}(\solvmfd;\R)$ is the de Rham cohomology of $\solvmfd$.
Suppose that $\solvmfd$ admits a left-invariant  symplectic structure.
Then the co-differential $\de^\Lambda$ is defined on $\wedge^{\bullet}\g^{\ast}$ and we can define the symplectic Bott-Chern and Aeppli cohomologies $ H^{\bullet}_{BC}(\g;\R) $ and $ H^{\bullet}_{A}(\g;\R) $ of the Lie algebra $\g$.
We consider the  maps
$$ H^{\bullet}_{BC}(\g;\R)\to  H^{\bullet}_{BC}(X;\R)  \qquad \text{ and } \qquad  H^{\bullet}_{A}(\g;\R) \to H^{\bullet}_{A}(X;\R)  $$
induced by the inclusion $\iota: \wedge^{\bullet}\g^{\ast}\hookrightarrow ^{\bullet}\solvmfd$.
In \cite{macri}, M. Macr\`i prove that if the map  $H^{\bullet}(\g;\R)\to H^{\bullet}_{dR}(\solvmfd;\R)$ is an isomorphism,  then the  maps
$$ H^{\bullet}_{BC}(\g;\R)\to  H^{\bullet}_{BC}(X;\R)  \qquad \text{ and } \qquad  H^{\bullet}_{A}(\g;\R) \to H^{\bullet}_{A}(X;\R)  $$ are also isomorphisms.
Hence, by Hattori's theorem in \cite{hattori}, if $G$ is completely solvable, then  the  maps
$$ H^{\bullet}_{BC}(\g;\R)\to  H^{\bullet}_{BC}(X;\R)  \qquad \text{ and } \qquad  H^{\bullet}_{A}(\g;\R) \to H^{\bullet}_{A}(X;\R)  $$ are  isomorphisms.
In this paper, we give a new proof of Macr\`i's result (Theorem \ref{thm:sympl-cohom-compl-solv}).
Moreover, in this paper,  we will  reach  a more general case. 
In fact, on a general solvmanifold $\solvmfd$, the  map  $H^{\bullet}(\g;\R)\to H^{\bullet}_{dR}(\solvmfd;\R)$ is not an isomorphism and we can not apply Macr\`i's result to general solvmanifolds.
In \cite{kasuya-jdg}, for any solvmanifold $\solvmfd$, the second author construct an explicit finite-dimensional sub-complex $A_{\Gamma}^{\bullet}\subseteq \wedge^{\bullet} \solvmfd\otimes_\R \C$ so that the inclusion induces a cohomology isomorphism $$H^{\bullet}(A_{\Gamma}^{\bullet})\cong  H^{\bullet}_{dR}(\solvmfd;\R)\otimes_\R \C.$$
In this paper, we show the following result.

\smallskip
\noindent{\bfseries Theorem (see Theorem \ref{thm:sympl-cohom-solvmfd}).}
{\itshape
Suppose that a left-invariant symplectic structure $\omega$ on a solvmanifold $\solvmfd$ satisfies $\omega\in A_{\Gamma}^{2}$.
Then, the finite-dimensional sub-complex $A^\bullet_{\Gamma} $ allows to compute the symplectic Bott-Chern and Aeppli  cohomologies of $\solvmfd$.
}
\smallskip

For a compact $2n$-dimensional  symplectic manifold $X$ with a symplectic form $\omega$, 
the cohomologies  $ H^{\bullet}_{BC}(X;\R) $ and $ H^{\bullet}_{A}(X;\R) $  relate to the {\em Hard Lefschetz Condition}, namely, for every $k\in\Z$ with $0\le k\le n$, the map $$\left[\omega^k\right] \cp \sspace \colon H^{n-k}_{dR}(X;\R) \to H^{n+k}_{dR}(X;\R)$$ is an isomorphism.
The Hard Lefschetz Condition is equivalent to the natural map $H^{\bullet}_{BC}(X;\R) \to H^{\bullet}_{dR}(X;\R)$ being injective: in such a case, we say that $X$ satisfies the {\em $\de\de^\Lambda$-Lemma}.

As similar to the complex case, \cite{angella-tomassini-3}, there is an inequality {\itshape à la} Fr\"olicher relating the dimensions of the symplectic Bott-Chern and Aeppli cohomologies and the Betti numbers, \cite{angella-tomassini-5}: for every $k\in\Z$,
$$ \dim_\R H^k_{BC}(X;\R) + \dim_\R H^k_{A}(X;\R) - 2\, \dim_\R H^k_{dR}(X;\R) \;\geq\; 0 \;; $$
furthermore, the equality holds for every $k\in\Z$ if and only if $X$ satisfies the $\de\de^\Lambda$-Lemma.

It is known that a non-toral nilmanifold does not satisfy the Hard Lefschetz Condition, \cite{benson-gordon-nilmanifolds}.
In Section \ref{sec:applications}, we compute the symplectic cohomologies for $4$-dimensional solvmanifolds, for $6$-dimensional nilmanifolds,
 and for the Nakamura manifold. 
 In this sense, the numbers $$\left\{\dim_\R H^{k}_{BC}(X;\R)+\dim_\R H^{k}_{A}(X;\R)-2\,\dim_\R H^{k}_{dR}(X;\R)\right\}_{k\in\Z}$$ provide a measure of the non-K\"ahlerianity of the nilmanifold $X$.

\medskip

We further consider symplectic cohomologies with values in a local system in Section \ref{sec:twisted-sympl-cohom}.
By using $\slf_{2}(\C)$-representations on bi-differential $\Z$-graded complexes, we study twisted Hard Lefschetz Condition and $D_\phi D_\phi^\Lambda$-Lemma.
Finally, in Section \ref{sec:twisted-sympl-cohom-solvmfd}, we investigate twisted symplectic cohomologies on solvmanifolds.
As similar to the non-twisted case, by the Hattori theorem, \cite{hattori}, and the Mostow theorem, \cite{mostow}, we compute the symplectic cohomologies of special solvmanifolds with values in a local system by using Lie algebras.
Moreover, considering the spaces of differential forms on the solvmanifolds with values  in certain local systems so that these spaces have structures of  differential graded algebras given in Hain's paper \cite{Hai}, we compute the symplectic cohomologies of these differential graded algebras by using the Sullivan minimal models constructed by the second author in \cite{kasuya-jdg}, see Theorem \ref{sym-min}, Theorem \ref{aep-min}.
In fact, the cohomology computation by $A^{\bullet}_{\Gamma}$ as above is a consequence of such results.
By these results, we compute the twisted symplectic cohomologies of Sawai's examples of symplectic solvmanifolds which satisfy the Hard Lefschetz Condition but do not satisfy the twisted Hard Lefschetz Condition.

In general, twisted cohomologies do not have self-duality and so, when considering the Lefschetz operators on twisted cohomologies, surjectivity and injectivity are not equivalent.
By this, as regards twisted Hard Lefschetz Conditions and $D_\phi D_\phi^\Lambda$-Lemma, we have differences between the twisted case and the non-twisted case.
We explain such a difference by using Sawai's examples.

\bigskip

\noindent{\sl Acknowledgments.}
The first author would like to warmly thank Adriano Tomassini for his constant support and encouragement, for his several advices, and for many inspiring conversations, and Federico A. Rossi for useful discussions. The second author would like to express his gratitude to Toshitake Kohno for helpful suggestions and stimulating discussions.
The authors would like also to thank the anonymous Referee of \cite{angella-kasuya-1} for her/his valuable suggestions that improved the presentation of this paper too, and Nicoletta Tardini and Adriano Tomassini for having pointed out some errors in the computations of Table \ref{table:dim-cohom-6-nilmfd} in the prevous version.

\section{Preliminaries and notations on cohomology computations}

In this section, we briefly recall the results in \cite{angella-kasuya-1}, about several cohomologies associated to a bounded double complex, respectively bi-differential $\Z$-graded complex, of $\C$-vector spaces. In particular, we are concerned with studying when such cohomologies can be recovered by means of a suitable (possibly finite-dimensional) sub-complex.
We also consider Lefschetz conditions and $\del\delbar$-Lemma for bi-differential $\Z$-graded complexes.

\begin{rem}
For a graded vector space $V^{\bullet}$ and  endomorphisms $A\in \mathrm{End}^{p}(V^{\bullet})$ and $B\in \mathrm{End}^{q}(V^{\bullet})$ of degrees $p$ and $q$ respectively, we denote $[A,B]=AB-(-1)^{pq}BA$.
\end{rem}

\subsection{Cohomologies of double complexes}
Consider a bounded double complex $\left(A^{\bullet,\bullet},\, \del,\, \delbar\right)$ of $\C$-vector spaces. Namely, $\del\in \End^{1,0}\left(A^{\bullet,\bullet}\right)$ and $\delbar\in \End^{0,1}\left(A^{\bullet,\bullet}\right)$ are such that $\del^2=\delbar^2=\left[\del,\delbar\right]=0$, and $A^{p,q}=\{0\}$ but for finitely-many $(p,q)\in\Z^2$.

One can consider several cohomologies associated to $\left(A^{\bullet,\bullet},\, \del,\, \delbar\right)$. More precisely, for $p\in\Z$ and for $q\in\Z$, one has the cohomologies
$$ H^{\bullet,q}_{\del}(A^{\bullet,\bullet}) \;:=\; H^{\bullet,q}\left(A^{\bullet,q},\, \del\right) \qquad \text{ and } \qquad H^{p,\bullet}_{\delbar}(A^{\bullet,\bullet}) \;:=\; H^{p,\bullet}\left(A^{p,\bullet},\, \delbar\right) \;. $$
By denoting the total complex associated to $\left(A^{\bullet,\bullet},\, \del,\, \delbar\right)$ by $\left(\Tot^\bullet \left( A^{\bullet,\bullet} \right) := \bigoplus_{p+q=\bullet} A^{p,q},\; \de:=\del+\delbar\right)$, one has the cohomology
$$ H^{\bullet}_{dR}(A^{\bullet,\bullet}) \;:=\; H^\bullet\left(\Tot^\bullet\left(A^{\bullet,\bullet}\right),\, \de\right) \;. $$
Furthermore, for $(p,q)\in\Z^2$, one can consider the \emph{Bott-Chern cohomology}, \cite{bott-chern},
$$ H^{p,q}_{BC}(A^{\bullet,\bullet}) \;:=\; H\left(A^{p-1,q-1} \stackrel{\del\delbar}{\longrightarrow} A^{p,q} \stackrel{\del+\delbar}{\longrightarrow} A^{p+1,q}\oplus A^{p,q+1}\right) \;, $$
and the \emph{Aeppli cohomology}, \cite{aeppli},
$$ H^{p,q}_{A}(A^{\bullet,\bullet}) \;:=\; H\left(A^{p-1,q}\oplus A^{p,q-1} \stackrel{\left(\del,\,	 \delbar\right)}{\longrightarrow} A^{p,q} \stackrel{\del\delbar}{\longrightarrow} A^{p+1,q+1}\right) \;. $$

\subsection{Fr\"olicher inequalities}
Let $\left(A^{\bullet,\bullet},\, \del,\, \delbar\right)$ be a bounded double complex of $\C$-vector spaces.

We recall that the natural filtrations induce naturally two spectral sequences such that
$$ {'E}_1^{\bullet_1,\bullet_2}\left(A^{\bullet,\bullet},\, \del,\, \delbar\right) \;\simeq\; H^{\bullet_2}\left(A^{\bullet_1,\bullet},\, \delbar\right) \;\Rightarrow\; H^{\bullet_1+\bullet_2}\left(\Tot^\bullet\left(A^{\bullet,\bullet}\right),\, \de\right) \;, $$
and 
$$ {''E}_1^{\bullet_1,\bullet_2}\left(A^{\bullet,\bullet},\, \del,\, \delbar\right) \;\simeq\; H^{\bullet_1}\left(A^{\bullet,\bullet_2},\, \del\right) \;\Rightarrow\; H^{\bullet_1+\bullet_2}\left(\Tot^\bullet\left(A^{\bullet,\bullet}\right),\, \de\right) \;, $$
see, e.g., \cite[\S2.4]{mccleary}, see also \cite[\S3.5]{griffiths-harris}, \cite[Theorem 1, Theorem 3]{cordero-fernandez-ugarte-gray}. In particular, one gets the Fr\"olicher inequalities:
$$ \dim_\C \Tot^\bullet H^{\bullet,\bullet}_{\delbar}\left(A^{\bullet,\bullet}\right) \;\geq\; \dim_\C H^\bullet_{dR}\left(A^{\bullet,\bullet}\right) \qquad \text{ and } \qquad \dim_\C \Tot^\bullet H^{\bullet,\bullet}_{\del}\left(A^{\bullet,\bullet}\right) \;\geq\; \dim_\C H^\bullet_{dR}\left(A^{\bullet,\bullet}\right) \;. $$

Furthermore, the first author and A. Tomassini proved in \cite{angella-tomassini-5} the following inequality {\itshape à la} Fr\"olicher for the Bott-Chern cohomology: assuming that $\dim_\C H^{\bullet,\bullet}_{\del}\left(A^{\bullet,\bullet}\right) < +\infty$ and $\dim_\C H^{\bullet,\bullet}_{\delbar}\left(A^{\bullet,\bullet}\right) < +\infty$, it holds that
$$ \dim_\C \Tot^\bullet H^{\bullet,\bullet}_{BC}\left(A^{\bullet,\bullet}\right) + \dim_\C \Tot^\bullet H^{\bullet,\bullet}_{A}\left(A^{\bullet,\bullet}\right) \;\geq\; 2 \, \dim_\C H^\bullet_{dR}\left(A^{\bullet,\bullet}\right) \;, $$
and the equality holds if and only if $\left(A^{\bullet,\bullet},\, \del,\, \delbar\right)$ satisfies the $\del\delbar$-Lemma, namely, the natural map $\Tot^\bullet H^{\bullet,\bullet}_{BC}\left(A^{\bullet,\bullet}\right) \to H^{\bullet}_{dR}\left(A^{\bullet,\bullet}\right)$ is injective.

\subsection{Bi-differential $\Z$-graded complexes}\label{subsec:doub}
In the symplectic case, the space of differential forms has just a $\Z$-graduation. Hence we consider the case of a bounded bi-differential $\Z$-graded complex $\left(A^{\bullet},\, \del,\, \delbar\right)$ of $\K$-vector spaces, where $\K\in\left\{\R,\, \C\right\}$. Namely, $A^\bullet$ is a $\Z$-graded $\K$-vector space endowed with $\del\in \End^{1}\left(A^{\bullet}\right)$ and $\delbar\in \End^{-1}\left(A^{\bullet}\right)$ such that $\del^2=\delbar^2=\left[\del,\delbar\right]=0$, and $A^{k}=\{0\}$ but for finitely-many $k\in\Z$.

Define
$$
H^{\bullet}_{dR}(A^{\bullet}) \;:=\; \frac{\ker(\del+\delbar)}{\imm(\del+\delbar)} \;,
\qquad
H^{\bullet}_{\del}(A^{\bullet}) \;:=\; \frac{\ker\del}{\imm\del} \;,
\qquad
H^{\bullet}_{\delbar}(A^{\bullet}) \;:=\; \frac{\ker\delbar}{\imm\delbar} \;,
$$
and
$$
H^{\bullet}_{BC}(A^{\bullet}) \;:=\; \frac{\ker\del\cap\ker\delbar}{\imm\del\delbar} \;,
\qquad
H^{\bullet}_{A}(A^{\bullet}) \;:=\; \frac{\ker\del\delbar}{\imm\del+\imm\delbar} \;.
$$

\medskip

Starting with a bi-differential $\Z$-graded complex, one can construct a double complex. Following \cite[\S1.3]{brylinski}, \cite[\S4.2]{cavalcanti}, see \cite[\S II.2]{goodwillie}, \cite[\S II]{connes}, take an infinite cyclic multiplicative group $\left\{ \beta^m \st m \in \Z \right\}$, and define the $\Z^2$-graded $\K$-vector space
$$ \Doub^{\bullet_1,\bullet_2} A^\bullet \;:=\; A^{\bullet_1-\bullet_2} \otimes_\K \K\beta^{\bullet_2} $$
endowed with
$$ \del\otimes_\K\id \;\in\; \End^{1,0}\left(\Doub^{\bullet,\bullet} A^\bullet \right) \qquad \text{ and } \qquad \delbar\otimes_\K\beta \;\in\; \End^{0,1}\left(\Doub^{\bullet,\bullet} A^\bullet \right) \;, $$
which satisfy $\left(\del\otimes_\K\id\right)^2=\left(\delbar\otimes_\K\beta\right)^2=\left[\del\otimes_\K\id,\, \delbar\otimes_\R\beta\right]=0$.

For $\sharp \in \left\{ dR, \, \del, \, \delbar, \, BC, \, A \right\}$, there are natural isomorphisms of $\K$-vector spaces, \cite[Lemma 1.5]{angella-tomassini-5}:
$$ H^{\bullet_1,\bullet_2}_{\sharp} \left(\Doub^{\bullet,\bullet}A^{\bullet}\right) \;\simeq\; H_{\sharp}^{\bullet_1 - \bullet_2}\left(A^\bullet\right) \otimes_\K \K\beta^{\bullet_2} \;. $$

\subsection{Cohomology computations}

We recall the following result in \cite{angella-kasuya-1}, concerning the relations between the Bott-Chern cohomology of a double complex and the Bott-Chern cohomology of a suitable sub-complex.

\begin{thm}[{\cite[Theorem 1.3]{angella-kasuya-1}}]\label{surjBC}
 Let $\left(A^{\bullet,\bullet},\, \del,\, \delbar\right)$ be a bounded double complex of $\C$-vector spaces, and let $\left(C^{\bullet,\bullet},\, \del,\, \delbar\right) \hookrightarrow \left(A^{\bullet,\bullet},\, \del,\, \delbar\right)$ be a sub-complex.
 Suppose that
 \begin{enumerate}
  \item\label{item:thm-surg-hp-1} the induced map $H^{\bullet,\bullet}_{\delbar}\left(C^{\bullet,\bullet}\right) \to H^{\bullet,\bullet}_{\delbar}\left(A^{\bullet,\bullet}\right)$ is an isomorphism,
  \item\label{item:thm-surg-hp-2} the induced map $H^{\bullet,\bullet}_{\del}\left(C^{\bullet,\bullet}\right) \to H^{\bullet,\bullet}_{\del}\left(A^{\bullet,\bullet}\right)$ is an isomorphism, and
  \item\label{item:thm-surg-hp-3} for any $(p,q)\in\Z^2$, the induced map
        $$ \frac{\ker\left(\de\colon \Tot^{p+q}\left(C^{\bullet,\bullet}\right) \to \Tot^{p+q+1}\left(C^{\bullet,\bullet}\right) \right) \cap C^{p,q}}{\imm\left(\de\colon \Tot^{p+q-1}\left(C^{\bullet,\bullet}\right) \to \Tot^{p+q}\left(C^{\bullet,\bullet}\right)\right)} \to \frac{\ker\left(\de\colon \Tot^{p+q}\left(A^{\bullet,\bullet}\right) \to \Tot^{p+q+1}\left(A^{\bullet,\bullet}\right) \right) \cap A^{p,q}}{\imm\left(\de\colon \Tot^{p+q-1}\left(A^{\bullet,\bullet}\right) \to \Tot^{p+q}\left(A^{\bullet,\bullet}\right)\right)} $$
        is surjective.
 \end{enumerate}
 Then the induced map $H^{\bullet,\bullet}_{BC}\left(C^{\bullet,\bullet}\right) \to H^{\bullet,\bullet}_{BC}\left(A^{\bullet,\bullet}\right)$ is surjective.
\end{thm}

As a corollary, we get the following result concerning cohomology computations for bi-differential $\Z$-graded complexes.

\begin{cor}\label{ives}
Let $\left(A^{\bullet},\, \del,\, \delbar\right)$ be  a bounded bi-differential $\Z$-graded complex and $B^{\bullet}\subseteq A^{\bullet}$ be a bi-differential $\Z$-graded sub-complex.
Suppose that:
\begin{enumerate}
\item the cohomologies $H^{\bullet}_{\del}(A^\bullet)$ and $H^{\bullet}_{\delbar}(A^\bullet)$, and $H^{\bullet}_{\del}(B^\bullet)$ and $H^{\bullet}_{\delbar}(B^\bullet)$ are finite-dimensional;
\item the inclusion $\iota \colon B^{\bullet}\hookrightarrow A^{\bullet}$ induces the isomorphisms $H^{\bullet}_{\del}(B^{\bullet}) \stackrel{\simeq}{\to} H^{\bullet}_{\del}(A^{\bullet})$ and $H^{\bullet}_{\delbar}(B^{\bullet}) \stackrel{\simeq}{\to} H^{\bullet}_{\delbar}(A^{\bullet})$;
\item there exists a map $\mu \colon A^{\bullet} \to B^{\bullet}$ of bi-differential $\Z$-graded complexes such that $\mu\circ \iota=\id_{B^\bullet}$.
\end{enumerate}
Then the inclusion $\iota \colon B^{\bullet}\hookrightarrow A^{\bullet}$ induces the isomorphism
\[ H^{\bullet}_{BC}(B^{\bullet}) \stackrel{\simeq}{\to} H^{\bullet}_{BC}(A^{\bullet}) \;. \]
\end{cor}

\begin{proof}
We have the induced map $\iota \colon H^{\bullet}_{BC}(B^\bullet)\to H^{\bullet}_{BC}(A^\bullet)$ and $\mu \colon H^{\bullet}_{BC}(A^\bullet) \to H^{\bullet}_{BC}(B^\bullet)$ such that $\mu \circ \iota = \id$.
Hence $\iota \colon H^{\bullet}_{BC}(B^\bullet) \to H^{\bullet}_{BC}(A^\bullet)$ is injective.

We prove the surjectivity of the map $\iota\colon H^{\bullet}_{BC}(B^\bullet) \to H^{\bullet}_{BC}(A^\bullet)$ by using Theorem \ref{surjBC}.
Consider the double complex $( \Doub^{\bullet_1,\bullet_2} A^\bullet \;:=\; A^{\bullet_1-\bullet_2} \otimes_\K \K\beta^{\bullet_2},   \del\otimes_\K\id, \delbar\otimes_\K\beta )$, as in Section \ref{subsec:doub}, and the double sub-complex $\Doub^{\bullet,\bullet} B^\bullet$.
Then, by the assumption, the inclusion $\iota\otimes \id \colon \Doub^{\bullet,\bullet} B^\bullet \hookrightarrow \Doub^{\bullet,\bullet} A^\bullet $ induces the isomorphisms
$$ H^{\bullet,\bullet}_{ \del\otimes_\K\id}(\Doub^{\bullet,\bullet} B^\bullet) \stackrel{\simeq}{\to} H^{\bullet,\bullet}_{ \del\otimes_\K\id}(\Doub^{\bullet,\bullet} A^\bullet) $$
and
$$ H^{\bullet,\bullet}_{ \delbar\otimes_\K\beta}(\Doub^{\bullet,\bullet} B^\bullet) \stackrel{\simeq}{\to} H^{\bullet,\bullet}_{ \delbar\otimes_\K\beta}(\Doub^{\bullet,\bullet} A^\bullet) \;. $$
Considering the spectral sequences of double complexes, by \cite[Theorem 3.5]{mccleary}, the inclusion $\iota\otimes \id \colon \Doub^{\bullet,\bullet} B^\bullet \hookrightarrow \Doub^{\bullet,\bullet} A^\bullet $ induces the isomorphism
\[H^{\bullet}_{dR} ({\Tot}^{\bullet}\Doub^{\bullet,\bullet} B^\bullet) \stackrel{\simeq}{\to} H^{\bullet}_{dR} ({\Tot}^{\bullet}\Doub^{\bullet,\bullet} A^\bullet) \;. 
\]

We prove that, for any $(p,q)\in\Z^2$, the induced map
 \begin{eqnarray*}
\iota\otimes \id\colon \lefteqn{ \frac{\ker\left(\de\colon \Tot^{p+q}\left(\Doub^{\bullet,\bullet} B^{\bullet} \right) \to \Tot^{p+q+1}\left(\Doub^{\bullet,\bullet}\wedge 
B^{\bullet}\right) \right) \cap \Doub^{p,q}B^{\bullet}}{\imm\left(\de\colon \Tot^{p+q-1}\left(\Doub^{\bullet,\bullet}B^{\bullet}\right) \to \Tot^{p+q}\left(\Doub^{\bullet,\bullet}B^{\bullet}\right)\right)} } \\[5pt]
 &\to& \frac{\ker\left(\de\colon \Tot^{p+q}\left(\Doub^{\bullet,\bullet}A^{\bullet}\right) \to \Tot^{p+q+1}\left(\Doub^{\bullet,\bullet}A^{\bullet}\right) \right) \cap \Doub^{p,q}A^{\bullet}}{\imm\left(\de\colon \Tot^{p+q-1}\left(\Doub^{\bullet,\bullet}A^{\bullet}\right) \to \Tot^{p+q}\left(\Doub^{\bullet,\bullet}A^{\bullet}\right)\right)}
 \end{eqnarray*}
 is surjective.

Consider the map ${\mu}\otimes \id \colon \Doub^{\bullet,\bullet} B^\bullet \rightarrow \Doub^{\bullet,\bullet} A^\bullet$.
This map is a homomorphism of double complexes.
Then, by the assumption that $\mu\circ\iota=\id$, for the induced map
\[\mu\otimes \id \colon H^{\bullet}_{dR} \left({\Tot}^{\bullet}\Doub^{\bullet,\bullet} A^\bullet\right) \to H^{\bullet}_{dR} \left({\Tot}^{\bullet}\Doub^{\bullet,\bullet} B^\bullet\right) \;, \]
we have $\left(\mu\otimes \id\right) \circ \left(\iota\otimes \id\right)={\id}$.
Since $\iota\otimes \id \colon H^{\bullet}_{dR} \left({\Tot}^{\bullet}\Doub^{\bullet,\bullet} B^\bullet\right) \stackrel{\simeq}{\to} H^{\bullet}_{dR} \left({\Tot}^{\bullet}\Doub^{\bullet,\bullet} A^\bullet\right)$ is an isomorphism, $\mu\otimes \id \colon H^{\bullet}_{dR} \left({\Tot}^{\bullet}\Doub^{\bullet,\bullet} A^\bullet\right) \to H^{\bullet}_{dR} \left({\Tot}^{\bullet}\Doub^{\bullet,\bullet} B^\bullet\right)$ is its inverse map.
Consider
\[ \frac{\ker\left(\de\colon \Tot^{p+q}\left(\Doub^{\bullet,\bullet} B^{\bullet} \right) \to \Tot^{p+q+1}\left(\Doub^{\bullet,\bullet}
B^{\bullet}\right) \right) \cap \Doub^{p,q}B^{\bullet}}{\imm\left(\de\colon \Tot^{p+q-1}\left(\Doub^{\bullet,\bullet}B^{\bullet}\right) \to \Tot^{p+q}\left(\Doub^{\bullet,\bullet}B^{\bullet}\right)\right)} \subseteq H^{\bullet}_{dR} \left({\Tot}^{\bullet}\Doub^{\bullet_1,\bullet_2} B^\bullet\right)
\]
and 
\[ \frac{\ker\left(\de\colon \Tot^{p+q}\left(\Doub^{\bullet,\bullet} A^{\bullet} \right) \to \Tot^{p+q+1}\left(\Doub^{\bullet,\bullet}
A^{\bullet}\right) \right) \cap \Doub^{p,q}A^{\bullet}}{\imm\left(\de\colon \Tot^{p+q-1}\left(\Doub^{\bullet,\bullet}A^{\bullet}\right) \to \Tot^{p+q}\left(\Doub^{\bullet,\bullet}A^{\bullet}\right)\right)} \subseteq H^{\bullet}_{dR} \left({\Tot}^{\bullet}\Doub^{\bullet_1,\bullet_2} A^\bullet\right) \;.
\]
By $\mu\otimes\id (\Doub^{p,q}A^{\bullet})\subseteq \Doub^{p,q}B^{\bullet}$, 
the induced map
 \begin{eqnarray*}
\iota\otimes \id\colon \lefteqn{ \frac{\ker\left(\de\colon \Tot^{p+q}\left(\Doub^{\bullet,\bullet} B^{\bullet} \right) \to \Tot^{p+q+1}\left(\Doub^{\bullet,\bullet}\wedge 
B^{\bullet}\right) \right) \cap \Doub^{p,q}B^{\bullet}}{\imm\left(\de\colon \Tot^{p+q-1}\left(\Doub^{\bullet,\bullet}B^{\bullet}\right) \to \Tot^{p+q}\left(\Doub^{\bullet,\bullet}B^{\bullet}\right)\right)} } \\[5pt]
 &\to& \frac{\ker\left(\de\colon \Tot^{p+q}\left(\Doub^{\bullet,\bullet}A^{\bullet}\right) \to \Tot^{p+q+1}\left(\Doub^{\bullet,\bullet}A^{\bullet}\right) \right) \cap \Doub^{p,q}A^{\bullet}}{\imm\left(\de\colon \Tot^{p+q-1}\left(\Doub^{\bullet,\bullet}A^{\bullet}\right) \to \Tot^{p+q}\left(\Doub^{\bullet,\bullet}A^{\bullet}\right)\right)}
 \end{eqnarray*}
is an isomorphism with the inverse map 
 \begin{eqnarray*}
\mu \otimes \id\colon \lefteqn{ \frac{\ker\left(\de\colon \Tot^{p+q}\left(\Doub^{\bullet,\bullet} A^{\bullet} \right) \to \Tot^{p+q+1}\left(\Doub^{\bullet,\bullet}\wedge 
A^{\bullet}\right) \right) \cap \Doub^{p,q}A^{\bullet}}{\imm\left(\de\colon \Tot^{p+q-1}\left(\Doub^{\bullet,\bullet}A^{\bullet}\right) \to \Tot^{p+q}\left(\Doub^{\bullet,\bullet}A^{\bullet}\right)\right)} } \\[5pt]
 &\to& \frac{\ker\left(\de\colon \Tot^{p+q}\left(\Doub^{\bullet,\bullet}B^{\bullet}\right) \to \Tot^{p+q+1}\left(\Doub^{\bullet,\bullet}B^{\bullet}\right) \right) \cap \Doub^{p,q}B^{\bullet}}{\imm\left(\de\colon \Tot^{p+q-1}\left(\Doub^{\bullet,\bullet}B^{\bullet}\right) \to \Tot^{p+q}\left(\Doub^{\bullet,\bullet}B^{\bullet}\right)\right)}.
 \end{eqnarray*}

By Theorem \ref{surjBC}, for any $(p,q)\in\Z^2$, the induced map
\[\iota\otimes \id\colon H_{BC}^{p - q}\left(B^\bullet\right) \otimes_\K \K\beta^{q} \simeq H^{p,q}_{BC}(\Doub^{\bullet,\bullet}B^{\bullet})\to H^{p,q}_{BC}(\Doub^{\bullet,\bullet}A^{\bullet}) \simeq H_{BC}^{p - q}\left(A^\bullet\right) \otimes_\K \K\beta^{q}
\]
 is surjective and so $H^{p,q}_{BC}(B^{\bullet})\to H^{p,q}_{BC}(A^{\bullet})$ is surjective.
Hence the corollary follows.
\end{proof}

\subsection{Hard Lefschetz condition for bi-differential $\Z$-graded complexes}
We recall here some definitions and results concerning $\slf_2(\C)$-representations, and Hard Lefschetz Condition for bi-differential $\Z$-graded complexes; we refer to \cite{yan}.

\medskip

Consider a (possibly non-finite-dimensional) $\C$-vector space $V$. Let $\phi\colon \slf_{2}(\C) \to  \End(V)$ be a representation of the Lie algebra $\slf_{2}(\C)$.
Take the basis
\[ X \;=\; \left(
\begin{array}{cc}
0&1\\
0&0
\end{array}
\right) \;, \qquad Y \;=\; \left(
\begin{array}{cc}
0&0\\
1&0
\end{array}
\right) \qquad \text{ and } \qquad Z \;=\; \left(
\begin{array}{cc}
-1&0\\
0&1
\end{array}
\right)
\]
of $\slf_{2}(\C)$.
We recall that $\phi$ is said of {\em finite $Z$-spectrum}, \cite[Definition 2.2]{yan}, if
\begin{enumerate}
\item $V$ can be decomposed as the direct sum of the eigenspaces of $\phi(Z)$, and
\item $\phi(Z)$ has only finitely many distinct eigenvalues.
\end{enumerate}
For $\slf_2(\C)$-representations of finite $Z$-spectrum, one has the following results, \cite[Corollary 2.4, Corollary 2.5, Corollary 2.6]{yan}:
\begin{enumerate}
\item all eigenvalues of $\phi(Z)$ are integers;
\item consider, for $k\in\Z$, the eigenspace $V_{k}$ of $\phi(Z)$ with respect to the eigenvalue $k$; for any $k\in\N$, the maps
$\phi(X)^{k}\colon V_{k}\to V_{-k}$ and $\phi(Y)^{k}\colon V_{-k}\to V_{k}$ are isomorphisms;
\item consider, for $k\in\Z$, the set $P_{k} = \left\{ v\in V_{k} \st \phi(X)v=0 \right\} = \left\{ v\in V_{k} \st \phi(Y)^{k+1}v=0 \right\}$ of primitive elements; then one has the following decompositions:
\begin{enumerate}
\item $V_{k} = P_{k} \oplus \imm \phi(X)$ for any $k\in\Z$;
\item $V_{k} = \bigoplus_{j\in\N} \phi(X)^j\left(P_{k+2j}\right)$ for any $k\in\N$;
\item $V_{-k} = \bigoplus_{j\in\N} \phi(X)^{k+j}\left(P_{k+2j}\right)$ for any $k\in\N\setminus\{0\}$.
\end{enumerate}
\end{enumerate}

\medskip

Now let $\left(A^{\bullet},\, \del,\, \delbar\right)$ be a bounded bi-differential $\Z$-graded complex such that $A^{k}=\{0\}$ for $k<0$ or $k>2n$, for some $n\in\N\setminus\{0\}$.
We define $H\in {\rm End}^{0}(A^{\bullet})$ as $H=\sum_{k\in \Z} (n-k)\pi_{A^{k}}$ where  $\pi_{A^{k}}:A^{\bullet}\to A^{k}$ denotes the projection.
We suppose that we have operators $L\in {\rm End}^{2}(A^{\bullet})$ and $\Lambda\in {\rm End}^{-2}(A^{\bullet})$
satisfying the following relations:
\begin{equation}\label{eq:comm-del-sympl}
\left[\del,L\right] \;=\; 0 \;, \qquad \left[\delbar, L\right] \;=\; -\del \;, \qquad  \left[\delbar, \Lambda\right]=0 \;, \qquad \left[\del, \Lambda\right] \;=\; \delbar \;,
\end{equation}
and
\begin{equation}\label{eq:sl2}
\left[\Lambda,L\right] \;=\; H \;, \qquad \left[L, H\right] \;=\; 2L \;, \qquad \left[\Lambda, H\right] \;=\; -2\Lambda \;.
\end{equation}
Then we have the representation $\phi\colon \slf_{2}(\C)\to \End(A^{\bullet})$ of finite $Z$-spectrum given by
$$ \phi(X) \;=\; \Lambda \;, \qquad \phi(Y) \;=\; L \;, \qquad \text{ and } \phi(Z) \;=\; H \;. $$
Hence, in particular, $L^{k}\colon A^{n-k}\to A^{n+k}$ is an isomorphism, for any $k\in\{0,\ldots,n\}$.

An element $\alpha\in A^{\bullet}$ is called {\em harmonic} if $\del\alpha=0$ and $\delbar \alpha=0$.
Let $H_{\text{hr}}^{\bullet}(A^{\bullet})$ be the space of all the harmonic elements.
Then, by the above relations, $H_{\text{hr}}^{\bullet}(A^{\bullet})$ is a $\slf_{2}(\C)$-submodule of finite $Z$-spectrum.
Hence, in particular, $L^{k}:H_{\text{hr}}^{n-k}(A^{\bullet})\to H_{\text{hr}}^{n+k}(A^{\bullet})$ is an isomorphism, for any $k\in\{0,\ldots,n\}$.

The same argument in \cite[Section 3]{yan} still works in yielding the following result.

\begin{thm}[{see \cite[Theorem 0.1]{yan}}]\label{har}
Let $\left(A^{\bullet},\, \del,\, \delbar\right)$ be a bounded bi-differential $\Z$-graded complex such that $A^{k}=\{0\}$ for $k<0$ or $k>2n$, for some $n\in\N\setminus\{0\}$, and that \eqref{eq:comm-del-sympl} and \eqref{eq:sl2} hold.
The following conditions are equivalent:
\begin{enumerate}
\item the inclusion $H_{\text{hr}}^{\bullet}(A^{\bullet}) \hookrightarrow A^{\bullet}$ induces the surjection $H_{\text{hr}}^{\bullet}(A^{\bullet})\to H_{\del}^{\bullet}(A^{\bullet})$;
\item for any $k\in\{0,\ldots,n\}$, the induced map $L^{k}\colon H^{n-k}_{\del}(A^{\bullet})\to H^{n+k}_{\del}(A^{\bullet})$ is surjective.
\end{enumerate}
\end{thm}

\begin{rem}
Note that the surjectivity of $H_{\text{hr}}^{\bullet}(A^{\bullet})\to H_{\del}^{\bullet}(A^{\bullet})$ does not imply the injectivity of $L^{k}\colon H^{n-k}_{\del}(A^{\bullet})\to H^{n+k}_{\del}(A^{\bullet})$.
In the second part of the above theorem, we can not state that $L^{k}\colon H^{n-k}_{\del}(A^{\bullet})\to H^{n+k}_{\del}(A^{\bullet})$ is an isomorphism, unless the cohomology $H^\bullet_{\del}(A^\bullet)$ is self-dual and finite-dimensional.
\end{rem}

In fact, the following result holds.

\begin{cor}\label{sel}
Let $\left(A^{\bullet},\, \del,\, \delbar\right)$ be a bounded bi-differential $\Z$-graded complex such that $A^{k}=\{0\}$ for $k<0$ or $k>2n$, for some $n\in\N\setminus\{0\}$, and that \eqref{eq:comm-del-sympl} and \eqref{eq:sl2} hold.
Furthermore, assume that the cohomology $H^{\bullet}_{\del}(A^{\bullet})$ is finite-dimensional and, for each $k\in\{0,\ldots,n\}$, we have $\dim H^{n-k}_{\del}(A^{\bullet})=\dim H^{n+k}_{\del}(A^{\bullet})$.
Then the following conditions are equivalent:
\begin{enumerate}
\item the inclusion $H_{\text{hr}}^{\bullet}(A^{\bullet}) \hookrightarrow A^{\bullet}$ induces the surjection $H_{\text{hr}}^{\bullet}(A^{\bullet})\to H_{\del}^{\bullet}(A^{\bullet})$;
\item for any $k\in\{0,\ldots,n\}$, the induced map $L^{k}\colon H^{n-k}_{\del}(A^{\bullet})\to H^{n+k}_{\del}(A^{\bullet})$ is an isomorphism.
\end{enumerate}
\end{cor}

\medskip

Arguing by induction as in \cite[Proposition 5.4]{cavalcanti}, thanks to the above relations for $\del$, $\delbar$, $L$, and $\Lambda$, we have the following result.

\begin{prop}[{see \cite[Proposition 5.4]{cavalcanti}}]\label{del12}
Let $\left(A^{\bullet},\, \del,\, \delbar\right)$ be a bounded bi-differential $\Z$-graded complex such that $A^{k}=\{0\}$ for $k<0$ or $k>2n$, for some $n\in\N\setminus\{0\}$, and that \eqref{eq:comm-del-sympl} and \eqref{eq:sl2} hold.
Suppose that for any $k \in \{0,\ldots,n\}$, the induced map $L^{k}\colon H^{n-k}_{\del}(A^{\bullet})\to H^{n+k}_{\del}(A^{\bullet})$ is an isomorphism.
Moreover, suppose that, for $\alpha\in A^{1}$, if $\del\delbar\alpha=0$, then $\delbar \alpha$ is $\del$-exact.
Then we have
$$\imm\delbar\cap \ker \del=\imm \del\cap \imm \delbar.
$$
\end{prop}

\subsection{$\del\delbar$-Lemma for bi-differential $\Z$-graded complexes}

In this section, we prove some results concerning $\del\delbar$-Lemma for bi-differential $\Z$-graded complexes, in relation with the Hard Lefschetz Condition.

\medskip

As for the following result, compare also \cite[Theorem 4.3]{cavalcanti} and \cite[Proposition 2]{fernandez-Munoz-Ugarte}.

\begin{prop}\label{ddlam}
Let $\left(A^{\bullet},\, \del,\, \delbar\right)$ be  a bounded bi-differential $\Z$-graded complex such that $A^{k}=\{0\}$ for $k<0$ or $k>2n$, for some $n\in\N\setminus\{0\}$.
Suppose that:
\begin{enumerate}
\item the inclusion $H_{\text{hr}}^{\bullet}(A^{\bullet})\subseteq A^{\bullet}$ induces the surjection $H_{\text{hr}}^{\bullet}(A^{\bullet})\to H_{\del}^{\bullet}(A^{\bullet})$;
\item $\imm\delbar\cap\ker\del=\imm\del\cap\ker \delbar$.
\end{enumerate}
Then we have $\imm \delbar\cap\ker\del=\imm \del\delbar=\imm \del\cap\ker \delbar$.
\end{prop}

\begin{proof}
For $s\in\N\setminus\{0\}$, we show that $\imm \delbar\cap\ker\del=\imm \del\delbar=\imm \del\cap\ker \delbar$ on $A^{2n-s}$.
We will prove this by induction on $s\in\N\setminus\{0\}$.

First we consider the case $s=0$.
By $\imm \delbar=\{0\}$, we have  $\imm \delbar\cap\ker\del = \{0\} = \imm \del\cap\ker \delbar$, and 
$\imm \del\delbar=\imm(- \delbar \del)= \{0\}$.

Then we consider the case $s=1$.
Let $\beta\in \imm \delbar\cap\ker\del\cap A^{2n-1} = \imm \del\cap\ker \delbar\cap A^{2n-1}$ such that $\beta=\delbar \alpha$ for some $\alpha\in A^{2n}$.
Then we have $\del\delbar\alpha=\del \beta=0$.
By the assumption and by $\del A^{2n}=0$, we have $\alpha_{\text{hr}}\in H_{\text{hr}}^{2n}(A^\bullet)$ such that $\alpha=\alpha_{\text{hr}}+\del\gamma$ for some $\gamma\in A^{2n-1}$.
Hence we have $\beta=\delbar\alpha_{\text{hr}}+\delbar\del\gamma=\delbar\del \gamma=\del\delbar (-\gamma)$.

Assume now that $\imm \delbar\cap\ker\del=\imm \del\delbar=\imm \del\cap\ker \delbar$ on $A^{2n-r}$ with $r\leq s$.
We need to prove $\imm \delbar\cap\ker\del=\imm \del\delbar=\imm \del\cap\ker \delbar$ on $A^{2n-s-1}$.
Let $\alpha_{2n-s-1}\in \imm \delbar\cap\ker\del \cap A^{2n-s-1} = \imm \del\cap\ker \delbar\cap A^{2n-s-1}$ such that $\alpha_{2n-s-1}=\delbar\alpha_{2n-s}$ for some $\alpha_{2n-s}\in A^{2n-s}$.
Then we have $0=\del\alpha_{2n-s-1}=\del\delbar\alpha_{2n-s}=-\delbar\del \alpha_{2n-s}$.
Hence $\del \alpha_{2n-s} \in \imm\del\delbar \cap A^{2n-s+1}$ follows from induction hypothesis and we have $\del \alpha_{2n-s}=\del\delbar\alpha_{2n-s+1}$ for some $\alpha_{2n-s+1}\in A^{2n-s+1}$.
By the assumption that $H^{2n-s}_{\text{hr}}(A^\bullet) \to H^{2n-s}_{\del}(A^\bullet)$ is surjective, 
we have that there exist $\beta_{2n-s}\in A^{2n-s}$ harmonic and $\gamma_{2n-s-1}\in A^{2n-s-1}$ such that 
\[\alpha_{2n-s}-\delbar \alpha_{2n-s+1}-\beta_{2n-s}=\del \gamma_{2n-s-1}.
\]
Hence we have
\[\alpha_{2n-s-1}=\delbar \alpha_{2n-s}=\delbar \del \gamma_{2n-s-1}=\del\delbar(-\gamma_{2n-s-1})\in \imm\del\delbar.
\]
Thus the proposition follows.
\end{proof}

As for the following result, compare also \cite[Theorem 4.2]{cavalcanti} and \cite[Theorem 2]{fernandez-Munoz-Ugarte}.

\begin{prop}\label{delin}
Let $\left(A^{\bullet},\, \del,\, \delbar\right)$ be  a bounded bi-differential $\Z$-graded complex such that $A^{k}=\{0\}$ for $k<0$ or $k>2n$, for some $n\in\N\setminus\{0\}$.
Suppose that
$$\imm\delbar\cap\ker\del=\imm\del\delbar=\imm\del\cap\ker \delbar \;. $$
Then the inclusion $H_{\text{hr}}^{\bullet}(A^{\bullet})\subseteq A^{\bullet}$ induces the surjection $H_{\text{hr}}^{\bullet}(A^{\bullet})\to H_{\del}^{\bullet}(A^{\bullet})$.
\end{prop}

\begin{proof}
For a $\del$-closed $k$-element $\alpha \in A^{k}$, we have
$\del\delbar\alpha=-\delbar\del\alpha=0$. 
By $\delbar\alpha\in  \imm \delbar\cap\ker\del=\imm \del\delbar$, we have $\beta\in A^{k-1}$ such that $\delbar \alpha=\del\delbar\beta$.
Hence $\alpha-\del\beta$ is a harmonic element which is cohomologous to $\alpha$.
\end{proof}

We summarize the contents of Proposition \ref{ddlam} and Proposition \ref{delin} in the following corollary.

\begin{cor}
Let $\left(A^{\bullet},\, \del,\, \delbar\right)$ be  a bounded bi-differential $\Z$-graded complex such that $A^{k}=\{0\}$ for $k<0$ or $k>2n$, for some $n\in\N\setminus\{0\}$.
Suppose that $\imm\delbar\cap\ker\del=\imm\del\cap\ker \delbar$.
The following conditions are equivalent:
\begin{enumerate}
 \item the inclusion $H_{\text{hr}}^{\bullet}(A^{\bullet})\subseteq A^{\bullet}$ induces the surjection $H_{\text{hr}}^{\bullet}(A^{\bullet})\to H_{\del}^{\bullet}(A^{\bullet})$;
 \item $\imm \delbar\cap\ker\del=\imm \del\delbar=\imm \del\cap\ker \delbar$
\end{enumerate}
\end{cor}

\section{Preliminaries and notations on symplectic structures}
In this section, we set the notations concerning symplectic structures on manifolds and symplectic Hodge theory, referring to, e.g., \cite{cannasdasilva, brylinski, mathieu, yan, cavalcanti, tseng-yau-1} for more details.

\subsection{Symplectic structures}\label{ope}
Let $X$ be a $2n$-dimensional compact manifold endowed with a \emph{symplectic structure} $\omega$, namely, a non-degenerate $\de$-closed $2$-form on $X$. The symplectic form $\omega$ yields the natural isomorphism
$$ I \colon TX \to T^*X \;, \qquad I(\sspace)(\ssspace) \;:=\; \omega(\sspace,\ssspace) \;.$$
Define the \emph{canonical Poisson bi-vector} associated to $\omega$ as
$$ \Pi \;:=\; \omega^{-1} \;:=\; \omega\left(I^{-1}\sspace,I^{-1}\ssspace\right) \;\in\; \wedge^2 TX \;, $$
and, for $k\in\N$, let $\left(\omega^{-1}\right)^k$ be the bi-linear form on $\wedge^k X$ defined, on the simple elements $\alpha^1\wedge\cdots \wedge \alpha^k\in\wedge^kX$ and $\beta^1\wedge\cdots \wedge \beta^k \in \wedge^kX$, as
$$
\left(\omega^{-1}\right)^k\left(\alpha^1\wedge\cdots \wedge \alpha^k,\, \beta^1\wedge\cdots \wedge \beta^k\right) \;:=\; \det\left(\omega^{-1}\left(\alpha^i,\, \beta^j\right)\right)_{i,j\in\{1,\ldots,k\}} \;.
$$
Define the \emph{symplectic-$\star$-operator}, \cite[\S2]{brylinski},
$$
\star_\omega\colon \wedge^\bullet X\to \wedge^{2n-\bullet}X
$$
by requiring that, for every $\alpha,\beta\in\wedge^kX$,
$$ \alpha\wedge\star_\omega \beta \;=\; \left(\omega^{-1}\right)^k\left(\alpha,\beta\right)\,\omega^n \;.$$

The operators
\begin{eqnarray*}
L \;\in\; \End^2\left(\wedge^\bullet X\right) \;, \quad && L(\alpha) \;:=\; \omega\wedge\alpha \;,\\[5pt]
\Lambda \;\in\; \End^{-2}\left(\wedge^\bullet X\right) \;, \quad && \Lambda(\alpha) \;:=\; -\iota_\Pi\alpha \;,\\[5pt]
H \;\in\; \End^0\left(\wedge^\bullet X\right) \;, \quad && H(\alpha) \;:=\; \sum_{k\in\Z} \left(n-k\right)\,\pi_{\wedge^kX}\alpha \;,
\end{eqnarray*}
yield an $\slf_2(\C)$-representation on $\wedge^\bullet X\otimes\C$, see, e.g., \cite[Corollary 1.6]{yan} (where $\iota_{\xi}\colon\wedge^{\bullet}X\to\wedge^{\bullet-2}X$ denotes the interior product with $\xi\in\wedge^2\left(TX\right)$, and $\pi_{\wedge^kX}\colon\wedge^\bullet X\to\wedge^kX$ denotes the natural projection onto $\wedge^kX$, for $k\in\Z$).

Define the \emph{symplectic co-differential operator} as
$$ \de^\Lambda \;:=\; \left[\de,\, \Lambda\right] \;\in\; \End^{-1}\left(\wedge^\bullet X \right) \;. $$
One has that
$$ \de^2 \;=\; \left(\de^\Lambda\right)^2 \;=\; \left[\de,\, \de^\Lambda\right] \;=\; 0 \;,$$
see, e.g., \cite[page 266, page 265]{koszul}, \cite[Proposition 1.2.3, Theorem 1.3.1]{brylinski}.

\subsection{Symplectic cohomologies}
Let $X$ be a $2n$-dimensional compact manifold endowed with a symplectic structure $\omega$.

By considering the bi-differential $\Z$-graded complex $\left(\wedge^\bullet X, \de, \de^\Lambda\right)$, one has the symplectic cohomologies $H^\bullet_{\sharp}\left(X;\R\right) := H^\bullet_{\sharp}\left(\wedge^\bullet X\right)$, for $\sharp\in\left\{dR, \de^\Lambda, BC, A\right\}$.
More precisely, other than the cohomologies
$$ H^\bullet_{dR}(X;\R) \;:=\; \frac{\ker\de}{\imm\de} \qquad \text{ and } \qquad H^\bullet_{\de^\Lambda}(X;\R) \;:=\; \frac{\ker\de^\Lambda}{\imm\de^\Lambda} \;, $$
one can define, following L.-S. Tseng and S.-T. Yau, \cite[\S3.2, \S3.3]{tseng-yau-1}, see also \cite{tseng-yau-2, tseng-yau-3},
$$ H^\bullet_{BC}(X;\R) \;:=\; \frac{\ker\de\cap\ker\de^\Lambda}{\imm\de\de^\Lambda} \qquad \text{ and } \qquad H^\bullet_{A}(X;\R) \;:=\; \frac{\ker\de\de^\Lambda}{\imm\de+\imm\de^\Lambda} \;. $$

In view of generalized complex geometry, \cite{gualtieri, cavalcanti, cavalcanti-impa}, these cohomologies are the symplectic counterpart of the Bott-Chern and Aeppli cohomologies for complex manifolds, \cite{tseng-yau-3}.

Note that $\star_\omega^2=\id_{\wedge^\bullet X}$, \cite[Lemma 2.1.2]{brylinski}, and that $\de^\Lambda\lfloor_{\wedge^kX} = \left(-1\right)^{k+1}\, \star_\omega \, \de \, \star_\omega$ for any $k\in\N$, \cite[Theorem 2.2.1]{brylinski}. In particular, it follows that the symplectic Hodge-$\star$-operator induces the isomorphism, \cite[Corollary 2.2.2]{brylinski},
$$ \star_\omega \colon H^\bullet_{dR}(X;\R) \stackrel{\simeq}{\to} H^{2n-\bullet}_{\de^\Lambda}(X;\R) \;. $$

\medskip

In \cite{tseng-yau-1}, L.-S. Tseng and S.-T. Yau developed a Hodge theory for the symplectic cohomologies. More precisely, fixed an almost-K\"ahler structure $\left(J, \omega, g:=\omega\left(\sspace, \, J\ssspace\right)\right)$ on $X$, they defined self-adjoint elliptic differential operators whose kernel is isomorphic to the above cohomologies, \cite[Proposition 3.3, Theorem 3.5, Theorem 3.16]{tseng-yau-1}. In particular, $X$ being compact, it follows that $\dim_\R H^\bullet_\sharp(X;\R) < +\infty$ for $\sharp\in\left\{dR,\, \de^\Lambda,\, BC,\, A\right\}$, \cite[Corollary 3.6, Corollary 3.17]{tseng-yau-1}. As another consequence, the Hodge-$*$-operator $*_g\colon \wedge^\bullet X \to \wedge^{2n-\bullet}X$ induces the isomorphism, \cite[Corollary 3.25]{tseng-yau-1},
$$ *_g\colon H^\bullet_{BC}(X;\R) \stackrel{\simeq}{\to} H^{2n-\bullet}_{A}(X;\R) \;. $$

\subsection{Hard Lefschetz Condition}
Several special cohomological properties can be defined on symplectic manifolds. More precisely, a compact $2n$-dimensional manifold $X$ endowed with a symplectic structure $\omega$ is said to satisfy:
\begin{itemize}
 \item the {\em Hard Lefschetz Condition} if, for any $k\in\N$, the map $\left[\omega^k\right]\cp\sspace \colon H^{n-k}_{dR}(X;\R) \to H^{n+k}_{dR}(X;\R)$ is an isomorphism;
 \item the {\em Brylinski conjecture}, \cite[Conjecture 2.2.7]{brylinski}, if every de Rham cohomology class admits a $\de$-closed $\de^\Lambda$-closed representative, namely, the natural map $H^{\bullet}_{BC}(X;\R) \to H^{\bullet}_{dR}(X;\R)$ induced by the identity is surjective;
 \item the {\em $\de\de^\Lambda$-Lemma} if every $\de$-exact $\de^\Lambda$-closed form is also $\de\de^\Lambda$-exact, namely, if the natural map $H^{\bullet}_{BC}(X;\R) \to H^{\bullet}_{dR}(X;\R)$ induced by the identity is injective.
\end{itemize}

By \cite[Corollary 2]{mathieu}, \cite[Theorem 0.1]{yan}, \cite[Proposition 1.4]{merkulov}, \cite{guillemin}, \cite[Proposition 3.13]{tseng-yau-1}, \cite[Theorem 5.4]{cavalcanti}, it follows that, for compact symplectic manifolds, the Hard Lefschetz Condition, the Brylinski conjecture, and the $\de\de^\Lambda$-Lemma are equivalent properties. In this case, it follows that the natural maps
$$
\xymatrix{
& H^{\bullet}_{BC}(X;\R) \ar[dl] \ar[dr] & \\
H^{\bullet}_{dR}(X;\R) \ar[dr] & & H^{\bullet}_{\de^\Lambda}(X;\R) \ar[dl] \\
& H^{\bullet}_{A}(X;\R) & \\
}
$$
induced by the identity are actually isomorphisms.

(See also, e.g., \cite[Proposition 3.13]{tseng-yau-1}, \cite{bock}, \cite[\S3, Theorem 4]{macri}, \cite[Remark 2.3]{angella-tomassini-4}, \cite[Theorem 4.4]{angella-tomassini-5}, \cite[\S8]{kasuya-jdg}, \cite[\S5]{kasuya-osaka}.)

\section{Symplectic cohomologies for solvmanifolds}

In this section, we apply the results in \cite{angella-kasuya-1} in order to provide tools for the computations of the symplectic cohomologies for solvmanifolds. In particular, we recover a theorem by M. Macrì, \cite{macri}, for completely-solvable solvmanifolds, see Theorem \ref{thm:sympl-cohom-compl-solv}, and we extend the result to the general case in Theorem \ref{thm:sympl-cohom-solvmfd}. Such results will be used in Section \ref{sec:applications} to investigate explicit examples.

\subsection{Notations}
In order to fix notations, let $X = \solvmfd$ be a \emph{solvmanifold} (that is, a compact quotient of a connected simply-connected solvable Lie group by a co-compact discrete subgroup) endowed with a $G$-left-invariant symplectic structure $\omega$. Denote the Lie algebra associated to $G$ by $\g$: it is endowed with the linear symplectic structure $\omega\in\wedge^2\g^\ast$. Denote the complexification of $\g$ by $\g_\C := \g \otimes_\R \C$.

Given a bi-differential $\Z$-graded sub-complex $\left( A^{\bullet},\, \de,\, \de^\Lambda \right) \hookrightarrow \left( \wedge^\bullet X,\, \de,\, \de^\Lambda \right)$, consider, for $\sharp \in \left\{ dR,\, \de^\Lambda,\, BC,\, A \right\}$,
$$ \iota\colon H^{\bullet}_{\sharp}(A^\bullet) \to H^{\bullet}_{\sharp}(X;\R) \;. $$

In particular, by means of left-translations, one has the $\Z$-graded $\R$-vector sub-space $\iota\colon \wedge^\bullet\g^\ast \hookrightarrow \wedge^\bullet X$. Since $\omega$ is $G$-left-invariant, the space $\wedge^\bullet\g^\ast$ is endowed with the (restrictions of the) differentials $\de$ and $\de^\Lambda$. In particular, $\left(\wedge^\bullet \mathfrak{g}^* ,\, \de ,\, \de^\Lambda \right)$ is a bi-differential $\Z$-graded sub-complex of $\left( \wedge^\bullet X ,\, \de ,\, \de^\Lambda \right)$.
For $\sharp \in \left\{ dR,\, \de^\Lambda,\, BC,\, A \right\}$, denote
$$ \iota\colon H^{\bullet}_{\sharp}(\mathfrak{g};\R) \;:=\; H^{\bullet}_{\sharp}(\wedge^\bullet\mathfrak{g}^*) \to H^\bullet_{\sharp}\left(X;\R\right) \;. $$

\subsection{Subgroups of symplectic cohomologies}

We firstly note the following result, as a consequence of the symplectic Hodge theory developed by L.-S. Tseng and S.-T. Yau in \cite{tseng-yau-1}.

\begin{cor}\label{cor:inj-sympl-solv}
 Let $\solvmfd$ be a $2n$-dimensional solvmanifold endowed with a $G$-left-invariant symplectic structure $\omega$. Let $\left(A^\bullet,\, \de\right)$ be a sub-complex of $\left(\wedge^\bullet X,\, \de\right)$ such that $A^2 \ni \omega$, and suppose that there exists an almost-K\"ahler structure $\left(J, \omega, g\right)$ on $X$ such that the Hodge-$*$-operator associated to $g$ satisfies $\left.*_g\right\lfloor_{A^\bullet} \colon A^\bullet \to A^{2n-\bullet}$. Then, for $\sharp\in\left\{dR,\, \de^\Lambda,\, BC,\, A\right\}$, the natural map
 $$ \iota\colon H^\bullet_{\sharp}\left(A^\bullet\right) \to H^\bullet_{\sharp}\left(X;\R\right) $$
 is injective.
\end{cor}

\begin{proof}
 Take an almost-K\"ahler structure $(J,\omega,g)$ as in the statement. In particular, $J$ is an almost-complex structure being compatible with $\omega$, that is, $g:=\omega(\sspace,\, J\ssspace)$ is a $J$-Hermitian metric with associated fundamental form $\omega$. Consider the Hodge-$*$-operator $*_g\colon \wedge^{\bullet}X \to \wedge^{2n-\bullet}X$ associated to $g$.
 By \cite[Theorem 3.5, Corollary 3.6]{tseng-yau-1}, the $4^{\text{th}}$-order self-adjoint differential operator
 $$ D_{\de+\de^\Lambda} \;:=\; \left(\de\de^\Lambda\right)\left(\de\de^\Lambda\right)^* + \left(\de\de^\Lambda\right)^*\left(\de\de^\Lambda\right) + \left(\de^*\de^\Lambda\right)\left(\de^*\de^\Lambda\right)^* + \left(\de^*\de^\Lambda\right)^*\left(\de^*\de^\Lambda\right) + \de^*\de+\left(\de^\Lambda\right)^*\de^\Lambda $$
 is elliptic, and it induces an orthogonal decomposition
 $$ \wedge^\bullet X \;=\; \ker D_{\de+\de^\Lambda} \stackrel{\perp}{\oplus} \de\de^\Lambda \wedge^\bullet X \stackrel{\perp}{\oplus} \left(\de^*\wedge^{\bullet+1}X + \left(\de^\Lambda\right)^*\wedge^{\bullet-1}X\right) \;, $$
 and hence the isomorphism
 $$ H^{\bullet}_{BC}(X;\R) \;\simeq\; \ker D_{\de+\de^\Lambda} \;, $$

 By the hypotheses, one has that $\de$ and $\de^\Lambda$, and $*_g$, restricts to $A^\bullet$. In particular, $D_{\de+\de^\Lambda}\lfloor_{A^\bullet}\colon A^\bullet \to A^\bullet$ induces an isomorphism
 $$ H^\bullet_{BC}(A^\bullet) \;\simeq\; \ker D_{\de+\de^\Lambda} \;. $$

 Hence, (as in \cite[Theorem 1.6]{angella-kasuya-1},) one gets the commutative diagram
 $$
 \xymatrix{
  \ker D_{\de+\de^\Lambda}\lfloor_{A^\bullet} \ar[r]^{\simeq} \ar@{^{(}->}[d] & H^{\bullet}_{BC}(A^\bullet) \ar[d] \\
  \ker D_{\de+\de^\Lambda} \ar[r]^{\simeq} \ar[r]^{\simeq} & H^{\bullet}_{BC}(X;\R)
 }
 $$
 from which it follows that the natural map $H^{\bullet}_{BC}(A^\bullet) \to H^{\bullet}_{BC}(X;\R)$ is injective.

 The theorem follows, by considering the differential elliptic operators $\left[\de,\de^*\right]$ and $\left[\de^\Lambda, \left(\de^\Lambda\right)^*\right]$ and, \cite[Theorem 3.16]{tseng-yau-1},
 \begin{eqnarray*}
  D_{\de\de^\Lambda} &:=& \left(\de\de^\Lambda\right)\left(\de\de^\Lambda\right)^* + \left(\de\de^\Lambda\right)^*\left(\de\de^\Lambda\right) + \left(\de\left(\de^\Lambda\right)^*\right)\left(\de\left(\de^\Lambda\right)^*\right)^* + \left(\de\left(\de^\Lambda\right)^*\right)^*\left(\de\left(\de^\Lambda\right)^*\right) \\[5pt]
  && + \de\de^* + \de^\Lambda\left(\de^\Lambda\right)^* \;,   
 \end{eqnarray*}
 such that $H^{\bullet}_{A}(X;\R)\simeq\ker D_{\de\de^\Lambda}$, \cite[Corollary 3.17]{tseng-yau-1},
 or by noting that $*_g D_{\de+\de^\Lambda}=D_{\de\de^\Lambda}*_g$, from which one has the isomorphism $*_g\colon H^{\bullet}_{BC}(X;\R)\stackrel{\simeq}{\to}H^{2n-\bullet}_{A}(X;\R)$, \cite[Lemma 3.23, Proposition 3.24, Corollary 3.25]{tseng-yau-1}.
\end{proof}

\subsection{Symplectic cohomologies for completely-solvable solvmanifolds}

By A. Hattori's theorem \cite[Corollary 4.2]{hattori}, if $G$ is \emph{completely-solvable}, (that is, for any $g\in G$, all the eigen-values of $\Ad g$ are real,) then the natural map $H^\bullet_{dR}\left(\g;\R\right) \to H^\bullet_{dR}\left(X;\R\right)$ is an isomorphism.

The following result states that, for a completely-solvable solvmanifold, the L.-S. Tseng and S.-T. Yau symplectic cohomologies can be computed using just left-invariant forms; we refer to \cite[Theorem 3]{macri} by M. Macrì for a different proof of the same result.

\begin{thm}[{see \cite[Theorem 3]{macri}}]\label{thm:sympl-cohom-compl-solv}
 Let $\solvmfd$ be a  solvmanifold endowed with a $G$-left-invariant symplectic structure $\omega$. 
Suppose that the natural map $H^\bullet_{dR}\left(\g;\R\right) \to H^\bullet_{dR}\left(X;\R\right)$ is an isomorphism.
Then, 
 for $\sharp\in\left\{\de^\Lambda,\, BC,\, A\right\}$, the natural map
 $$ \iota\colon H^\bullet_{\sharp}\left(\g;\R\right) \to H^\bullet_{\sharp}\left(X;\R\right) $$
 is an isomorphism.
\end{thm}

\begin{proof}
 We split the proof in the following steps.


 \paragrafo{1}{The symplectic $\de^\Lambda$-cohomology.}
 Since $\omega$ is $G$-left-invariant, then the symplectic-$\star$-operator $\star_\omega\colon \wedge^{\bullet}X \to \wedge^{\dim X - \bullet}X$ induces the isomorphism $\star_\omega\lfloor_{\wedge^\bullet \g^\ast} \colon \wedge^\bullet \g^\ast \to \wedge^{\dim_\R\g - \bullet} \g^\ast$. Hence, since $\left(\star_\omega\lfloor_{\wedge^\bullet \g^\ast}\right)^2 = \star_\omega^2\lfloor_{\wedge^\bullet \g^\ast} = \id_{\wedge^\bullet \g^\ast}$ and $\de^\Lambda\lfloor_{\wedge^\bullet \g^\ast} = \left(-1\right)^{k+1} \, \star_\omega\lfloor_{\wedge^\bullet \g^\ast} \, \de\lfloor_{\wedge^\bullet \g^\ast} \, \star_\omega\lfloor_{\wedge^\bullet \g^\ast}$ for any $k\in\N$, one has the isomorphism $\star_\omega \colon H^\bullet_{dR}(\g;\R) \stackrel{\simeq}{\to} H^{\dim_\R\g - \bullet}_{\de^\Lambda}(\g;\R)$. Therefore one gets the commutative diagram
 $$
 \xymatrix{
  H^\bullet_{dR}\left(\g;\R\right) \ar[r]^{\simeq}_{\iota} \ar[d]^{\simeq}_{\star_\omega} & H^\bullet_{dR}\left(X;\R\right) \ar[d]^{\simeq}_{\star_\omega} \\
  H^{\dim_\R\g - \bullet}_{\de^\Lambda}\left(\g;\R\right) \ar[r]_{\iota} & H^{\dim X - \bullet}_{\de^\Lambda}\left(X;\R\right) \;,
 }
 $$
 from which it follows that the natural map $\iota\colon H^\bullet_{\de^\Lambda}\left(\g;\R\right) \to H^\bullet_{\de^\Lambda}\left(X;\R\right)$ is an isomorphism.

 \paragrafo{2}{The symplectic Bott-Chern cohomology.}
 Apply Corollary \ref{ives} to $\iota \colon \left(\wedge^\bullet \g^\ast ,\, \de ,\, \de^\Lambda\right) \hookrightarrow \left(\wedge^\bullet X ,\, \de ,\, \de^\Lambda\right)$, where the map $\mu \colon \wedge^\bullet X \to \wedge^\bullet \g^\ast$ is the F.~A. Belgun symmetrization map, \cite[Theorem 7]{belgun}: namely, by \cite[Lemma 6.2]{milnor}, consider a $G$-bi-invariant volume form $\eta$ on $G$ such that $\int_X\eta=1$, and define
 $$ \mu\colon \wedge^\bullet X \otimes_\R \C \to \wedge^\bullet\g_\C^\ast \;, \qquad \mu(\alpha)\;:=\;\int_X \alpha\lfloor_x \, \eta(x) \;. $$

 \paragrafo{3}{The symplectic Aeppli cohomology.}
 Let $J$ be a $G$-left-invariant $\omega$-compatible almost-complex structure on $X$, (see, e.g., \cite[Proposition 12.6]{cannasdasilva},) and consider the $G$-left-invariant $J$-Hermitian metric $g:=\omega\left(\sspace, \, J\ssspace\right)$. By \cite[Corollary 3.25]{tseng-yau-1}, the Hodge-$*$-operator $*_g\colon \wedge^\bullet X \to \wedge^{2n-\bullet}X$ induces the isomorphism $*_g\colon H^\bullet_{BC}(X;\R) \stackrel{\simeq}{\to} H^{2n-\bullet}_{A}(X;\R)$, and, since $g$ is $G$-left-invariant, also the isomorphism $*_g\colon H^\bullet_{BC}(\g;\R) \stackrel{\simeq}{\to} H^{2n-\bullet}_{A}(\g;\R)$. Hence one has the commutative diagram
 $$
 \xymatrix{
  H^\bullet_{BC}\left(\g;\R\right) \ar[r]^{\simeq}_{\iota} \ar[d]^{\simeq}_{*_g} & H^\bullet_{BC}\left(X;\R\right) \ar[d]^{\simeq}_{*_g} \\
  H^{\dim_\R\g - \bullet}_{A}\left(\g;\R\right) \ar[r]_{\iota} & H^{\dim X - \bullet}_{A}\left(X;\R\right) \;,
 }
 $$
 from which it follows that also the natural map $\iota\colon H^\bullet_{A}\left(\g;\R\right) \to H^\bullet_{A}\left(X;\R\right)$ is an isomorphism.
\end{proof}

\section{Applications}\label{sec:applications}

In this section, as an application of Theorem \ref{thm:sympl-cohom-compl-solv}, we explicitly compute the symplectic cohomologies of some low-dimensional nilmanifolds and solvmanifolds.

\medskip

We recall that, by \cite[Theorem 4.4]{angella-tomassini-5}, for a compact manifold $X$ endowed with a symplectic structure, for any $k\in\Z$, the inequality
$$ \Delta^k \;:=\; \dim_\R H^{k}_{BC}(X;\R) + \dim_\R H^{k}_{A}(X;\R) - 2\, \dim_\R H^{k}_{dR}(X;\R) \;\geq\; 0 $$
holds. Furthermore, the equality holds for any $k\in\Z$ if and only if $X$ satisfies the Hard Lefschetz Condition.
Non-tori nilmanifolds never satisfy the Hard Lefschetz Condition by \cite[Theorem A]{benson-gordon-nilmanifolds}.
Hence, for nilmanifolds, the numbers $\left\{\Delta^k\right\}_{k\in\Z}$ provide a degree of non-K\"ahlerianity.
As regards Hard Lefschetz Condition for solvmanifolds, we refer to \cite{kasuya-jdg, kasuya-osaka}.
Finally, recall that, by \cite[Theorem 4.3]{tardini-tomassini}, for compact symplectic manifolds, it holds always $\Delta^1=0$.

\medskip

(As a matter of notations, by writing the structure equations of the Lie algebra $\g$ associated to a solvmanifold, we write, e.g., $(0,0,0,12)$: we mean that there exists a basis $\left\{e^1,e^2,e^3,e^4\right\}$ of $\g^*$ such that $\de e^1=\de e^2=\de e^3=0$ and $\de e^4=e^1\wedge e^2$; furthermore, we shorten $e^1\wedge e^2=:e^{12}$; we follow the notations in \cite{salamon, bock, cavalcanti-gualtieri}.)

\subsection{$4$-dimensional solvmanifolds}
According to \cite[Theorem 6.2, Table 2]{bock}, the compact $4$-dimensional manifolds that are diffeomorphic to a solvmanifold admitting a left-invariant symplectic structure are the following:
\begin{enumerate}[\itshape (a)]
 \item the {\em torus} $4\,\mathfrak{g}_{1}=(0,\, 0,\, 0,\, 0)$ endowed with the left-invariant symplectic structure $\omega:=e^{12}+e^{34}$: the Lie algebra is Abelian, and the symplectic structure satisfies the Hard Lefschetz Condition;
 \item the differentiable manifold underlying the {\em primary Kodaira surface} $\mathfrak{g}_{3.1}\oplus\mathfrak{g}_{1}=(0,\, 0,\, 0,\, 23)$ endowed with the left-invariant symplectic structure $\omega:=e^{12}+e^{34}$: the Lie algebra is nilpotent, and hence no left-invariant symplectic structure on such manifold satisfies the Hard Lefschetz Condition;
 \item the solvmanifold associated to $\mathfrak{g}_{1}\oplus\mathfrak{g}_{3.4}^{-1}=(0,\, 0,\, -23,\, 24)$ endowed with the left-invariant symplectic structure $\omega:=e^{12}+e^{34}$: the Lie algebra is completely-solvable, and the symplectic structure satisfies the Hard Lefschetz Condition;
 \item the differentiable manifold underlying the {\em hyper-elliptic surface}, whose associated Lie algebra is $\mathfrak{g}_{1}\oplus\mathfrak{g}_{3.5}^{0}=(0,\, 0,\, -24,\, 23)$, endowed with the left-invariant symplectic structure $\omega:=e^{12}+e^{34}$: it yields a K\"ahler structure on a solvmanifold, and hence the symplectic structure satisfies the Hard Lefschetz Condition;
 \item the manifold associated to $\mathfrak{g}_{4.1}=(0,\, 0,\, 12,\, 13)$ endowed with the left-invariant symplectic structure $\omega:=e^{14}+e^{23}$: the Lie algebra is nilpotent, and hence no left-invariant symplectic structure on such manifold satisfies the Hard Lefschetz Condition.
\end{enumerate}

\medskip

We compute the symplectic cohomologies of the above manifolds endowed with the indicated left-invariant symplectic structures. Note that, in case of the nilmanifolds {\itshape (a)}, {\itshape (b)}, and {\itshape (e)}, by Theorem \ref{thm:sympl-cohom-solvmfd}, see also \cite[Theorem 3]{macri}, it suffices to consider the left-invariant forms. On the other hand, since the symplectic structures on the manifolds {\itshape (c)} and {\itshape (d)} satisfy the Hard Lefschetz Condition, we know that the Bott-Chern, Aeppli, and de Rham cohomologies are all isomorphic.

In Table \ref{table:cohom-4-solvmfd} we list the harmonic representatives with respect to the left-invariant metric $g:=\sum_{j=1}^{4}e^j\odot e^j$, and in Table \ref{table:dim-coom-4-solvmfd} we summarize the dimensions of the symplectic cohomologies and the non-K\"ahlerianity degrees $\left\{\Delta^k\right\}_{k\in\{1,2,3\}}$.

\begin{center}
\begin{table}[ht]
 \centering
\resizebox{\textwidth}{!}{
\begin{tabular}{>{$\mathbf\bgroup}l<{\mathbf\egroup$} >{$\mathbf\bgroup}l<{\mathbf\egroup$} || >{$}p{3cm}<{$} >{$}p{5cm}<{$} >{$}p{4cm}<{$} || }
\toprule
 \multicolumn{2}{c||}{$H^\bullet_{\sharp}$} & H^1_{\sharp} & H^2_{\sharp} & H^3_{\sharp} \\
\toprule
\midrule[0.02em]
\multirow{2}{*}{$4\,\mathfrak{g}_{1}$} & \text{$H^\bullet_{BC}$} & \R\left\langle e^1,\; e^2,\; e^3,\; e^4 \right\rangle & \R\left\langle e^{12},\; e^{13},\; e^{14},\; e^{23},\; e^{24},\; e^{34} \right\rangle & \R\left\langle e^{123},\; e^{124},\; e^{134},\; e^{234} \right\rangle \\
 & \text{$H^\bullet_{A}$} & \R\left\langle e^1,\; e^2,\; e^3,\; e^4 \right\rangle & \R\left\langle e^{12},\; e^{13},\; e^{14},\; e^{23},\; e^{24},\; e^{34} \right\rangle & \R\left\langle e^{123},\; e^{124},\; e^{134},\; e^{234} \right\rangle \\
\midrule[0.02em]
\multirow{2}{*}{$\mathfrak{g}_{3.1}\oplus\mathfrak{g}_{1}$} & \text{$H^\bullet_{BC}$} & \R\left\langle e^1,\; e^2,\; e^3 \right\rangle & \R\left\langle e^{12},\; e^{13},\; e^{23},\; e^{24},\; e^{34} \right\rangle & \R\left\langle e^{123},\; e^{134},\; e^{234} \right\rangle \\
 & \text{$H^\bullet_{A}$} & \R\left\langle e^1,\; e^2,\; e^4 \right\rangle & \R\left\langle e^{12},\; e^{13},\; e^{14},\; e^{24},\; e^{34} \right\rangle & \R\left\langle e^{124},\; e^{134},\; e^{234} \right\rangle \\
\midrule[0.02em]
\multirow{2}{*}{$\mathfrak{g}_{3.4}^{-1}\oplus\mathfrak{g}_{1}$} & \text{$H^\bullet_{BC}$} & \R\left\langle e^1,\; e^2 \right\rangle & \R\left\langle e^{12},\; e^{34} \right\rangle & \R\left\langle e^{134},\; e^{234} \right\rangle \\
 & \text{$H^\bullet_{A}$} & \R\left\langle e^1,\; e^2 \right\rangle & \R\left\langle e^{12},\; e^{34} \right\rangle & \R\left\langle e^{134},\; e^{234} \right\rangle \\
\midrule[0.02em]
\multirow{2}{*}{$\mathfrak{g}_{3.5}^{0}\oplus\mathfrak{g}_{1}$} & \text{$H^\bullet_{BC}$} & \R\left\langle e^1,\; e^2 \right\rangle & \R\left\langle e^{12},\; e^{34} \right\rangle & \R\left\langle e^{134},\; e^{234} \right\rangle \\
 & \text{$H^\bullet_{A}$} & \R\left\langle e^1,\; e^2 \right\rangle & \R\left\langle e^{12},\; e^{34} \right\rangle & \R\left\langle e^{134},\; e^{234} \right\rangle \\
\midrule[0.02em]
\multirow{2}{*}{$\mathfrak{g}_{4.1}$} & \text{$H^\bullet_{BC}$} & \R\left\langle e^1,\; e^2 \right\rangle & \R\left\langle e^{12},\; e^{13},\; e^{14},\; e^{23} \right\rangle & \R\left\langle e^{123},\; e^{124} \right\rangle \\
 & \text{$H^\bullet_{A}$} & \R\left\langle e^3,\; e^4 \right\rangle & \R\left\langle e^{14},\; e^{23},\; e^{24},\; e^{34} \right\rangle & \R\left\langle e^{134},\; e^{234} \right\rangle \\
\bottomrule
\end{tabular}
}
\caption{The symplectic Bott-Chern and Aeppli cohomologies for $4$-dimensional solvmanifolds.}
\label{table:cohom-4-solvmfd}
\end{table}
\end{center}

\begin{table}[ht]
 \centering
\begin{tabular}{>{$\mathbf\bgroup}c<{\mathbf\egroup$} >{$\mathbf\bgroup}c<{\mathbf\egroup$} || >{$}c<{$} | >{$}c<{$} | >{$}c<{$} | >{$}c<{$} | >{$}c<{$} ||}
\toprule
\multicolumn{2}{c||}{\textnormal{$\dim_\R H_{\sharp}^{\bullet}$}} & 4\,\mathfrak{g}_{1} & \mathfrak{g}_{3.1}\oplus\mathfrak{g}_{1} & \mathfrak{g}_{3.4}^{-1}\oplus\mathfrak{g}_{1} & \mathfrak{g}_{3.5}^{0}\oplus\mathfrak{g}_{1} & \mathfrak{g}_{4.1} \\
& & (0,\, 0,\, 0,\, 0) & (0,\, 0,\, 0,\, 23) & (0,\, 0,\, -23,\, 24) & (0,\, 0,\, -24,\, 23) & (0,\, 0,\, 12,\, 13) \\
\toprule
\midrule[0.02em]
\multirow{4}{*}{k=1} & \text{$\dim_\R H^1_{dR}$} & 4 & 3 & 2 & 2 & 2 \\
& \text{$\dim_\R H^1_{BC}$} & 4 & 3 & 2 & 2 & 2 \\
& \text{$\dim_\R H^1_{A}$} & 4 & 3 & 2 & 2 & 2 \\
\hline
& \text{$\Delta^1$} & 0 & 0 & 0 & 0 & 0 \\
\midrule[0.02em]\midrule[0.02em]
\multirow{4}{*}{k=2} & \text{$\dim_\R H^2_{dR}$} & 6 & 4 & 2 & 2 & 2 \\
& \text{$\dim_\R H^2_{BC}$} & 6 & 5 & 2 & 2 & 4 \\
& \text{$\dim_\R H^2_{A}$} & 6 & 5 & 2 & 2 & 4 \\
\hline
& \text{$\Delta^2$} & 0 & 2 & 0 & 0 & 4 \\
\midrule[0.02em]\midrule[0.02em]
\multirow{4}{*}{k=3} & \text{$\dim_\R H^3_{dR}$} & 4 & 3 & 2 & 2 & 2 \\
& \text{$\dim_\R H^3_{BC}$} & 4 & 3 & 2 & 2 & 2 \\
& \text{$\dim_\R H^3_{A}$} & 4 & 3 & 2 & 2 & 2 \\
\hline
& \text{$\Delta^3$} & 0 & 0 & 0 & 0 & 0 \\
\bottomrule
\multicolumn{2}{c||}{\text{\bfseries{HLC?}}} & \text{\checkmark} & \times & \text{\checkmark} & \text{\checkmark} & \times \\
\bottomrule
\end{tabular}
\caption{Summary of the dimensions of symplectic Bott-Chern and Aeppli cohomologies for $4$-dimensional solvmanifolds.}
\label{table:dim-coom-4-solvmfd}
\end{table}

\subsection{$6$-dimensional nilmanifolds}
The $6$-dimensional nilmanifolds can be classified in terms of their Lie algebra, up to isomorphisms, in $34$ classes, according to V.~V. Morozov's classification, \cite{morozov}, see also \cite{magnin}.

As regards the complex geometry of $6$-dimensional nilmanifolds, S. Salamon proved in \cite{salamon} that just $18$ of these $34$ classes admit a complex structure. A complete classification, up to equivalence, of the left-invariant complex structures on $6$-dimensional nilmanifolds follows from the works by several authors, and was completed in \cite{ceballos-otal-ugarte-villacampa}. In \cite{angella-franzini-rossi, latorre-ugarte-villacampa}, the Bott-Chern cohomology for each of such structures is computed. In view of \cite{angella-tomassini-3}, such computations provide a measure of the non-K\"ahlerianity of $6$-dimensional nilmanifolds, but for nilmanifolds associated to the Lie algebra $\mathfrak{h}_7=\left(0,0,0,12,13,23\right)$ possibly.

As regards the symplectic geometry of $6$-dimensional nilmanifolds, M. Goze and Y. Khakimdjanov proved that $26$ of the $34$ classes of $6$-dimensional Lie algebras admit a symplectic structure, \cite{goze-khakimdjanov}.

Finally, G.~R. Cavalcanti and M. Gualtieri proved in \cite{cavalcanti-gualtieri} that every $6$-dimensional nilmanifolds admit a generalized-complex structure.

\medskip

By applying Theorem \ref{thm:sympl-cohom-compl-solv}, see also \cite[Theorem 3]{macri}, one can compute the symplectic cohomologies of the $6$-dimensional nilmanifolds endowed with a left-invariant symplectic structure. The results of the computations, which we performed with the aid of Sage \cite{sage}, are summarized in Table \ref{table:dim-cohom-6-nilmfd}.

\begin{sidewaystable}[pht]
\vspace{10cm}
 \centering
\resizebox{\textwidth}{!}{
\begin{tabular}{>{$}c<{$} >{$}c<{$} >{$}c<{$} >{$}c<{$} | >{$}c<{$} | >{$}c<{$} || >{$}c<{$} >{$}c<{$} >{$}c<{$} | >{$}c<{$} || >{$}c<{$} >{$}c<{$} >{$}c<{$} | >{$}c<{$} || >{$}c<{$} >{$}c<{$} >{$}c<{$} | >{$}c<{$} ||}
\toprule
 \multicolumn{6}{c||}{$\dim_\R H^\bullet_{\sharp}$} &  \multicolumn{4}{c||}{$k=1$} & \multicolumn{4}{c||}{$k=2$} & \multicolumn{4}{c||}{$k=3$} \\
 & \text{\cite{bock}} & \text{\cite{cavalcanti-gualtieri}} & & \omega & \omega = 12 + 34 + 56 & \dim_\R H^{1}_{dR} & \dim_\R H^{1}_{BC} & \dim_\R H^{1}_{A} & \Delta^1 & \dim_\R H^{2}_{dR} & \dim_\R H^{2}_{BC} & \dim_\R H^{2}_{A} & \Delta^2 & \dim_\R H^{3}_{dR} & \dim_\R H^{3}_{BC} & \dim_\R H^{3}_{A} & \Delta^3 \\
\toprule
\midrule[0.02em]
\mathfrak{g}_{6.N2} & 30 & 1 & \left(0,0,12,13,14,15\right) & 16+34-25 & \left(0,15,16,13,14,0\right) & 2 & 2 & 2 & 0 & 3 & 5 & 5 & 4 & 4 & 6 & 6 & 4 \\
\mathfrak{g}_{6.N19} & 31 & 3 & \left(0,0,12,13,14,23+15\right) & 16+24+34-25 & \left(0,63+15,16,13-16,14,0\right) & 2 & 2 & 2 & 0 & 3 & 5 & 5 & 4 & 4 & 6 & 6 & 4 \\
\mathfrak{g}_{6.N11} & 27 & 4 & \left(0,0,12,13,23,14\right) & 15+24+34-26 & \left(0,63,16,13-16,14,0\right) & 2 & 2 & 2 & 0 & 4 & 6 & 6 & 4 & 6 & 6 & 6 & 0 \\
\mathfrak{g}_{6.N18}^{1} & 29 & 5 & \left(0,0,12,13,23,14-25\right) & 15+24-35+16 & \left(35,35+13,0,15,0,54+23\right) & 2 & 2 & 2 & 0 & 4 & 6 & 6 & 4 & 6 & 6 & 6 & 0 \\
\mathfrak{g}_{6.N18}^{-1} & 28 & 6 & \left(0,0,12,13,23,14+25\right) & 15+24+35+16 & \left(53,31-35,0,51,0,54+32\right) & 2 & 2 & 2 & 0 & 4 & 6 & 6 & 4 & 6 & 6 & 6 & 0 \\
\mathfrak{g}_{6.N20} & 32 & 8 & \left(0,0,12,13,14+23,24+15\right) & 16+2\times34-25 & \left(0,64+15,2\times16,\frac{1}{2}\times13,14+\frac{1}{2}\times63,0\right) & 2 & 2 & 2 & 0 & 3 & 5 & 5 & 4 & 4 & 6 & 6 & 4 \\
\midrule[0.02em]
\mathfrak{g}_{6.N6} & 13 & 10 & \left(0,0,0,12,13,14+23\right) & 16-2\times34-25 & \left(0,13+\frac{1}{2}\times64,16,0,\frac{1}{2}\times14,0\right) & 3 & 3 & 3 & 0 & 6 & 9 & 9 & 6 & 8 & 11 & 11 & 6 \\
\mathfrak{g}_{6.N7} & 14 & 11 & \left(0,0,0,12,13,24\right) & 26+14+35 & \left(0,14,0,31,0,35\right) & 3 & 3 & 3 & 0 & 6 & 9 & 9 & 6 & 8 & 9 & 9 & 2 \\
\mathfrak{g}_{6.N1} & 12 & 12 & \left(0,0,0,12,13,14\right) & 16+24+35 & \left(0,14,0,13,0,15\right) & 3 & 3 & 3 & 0 & 6 & 9 & 9 & 6 & 8 & 10 & 10 & 4 \\
\mathfrak{g}_{6.N3} & 11 & 13 & \left(0,0,0,12,13,23\right) & 15+24+36 & \left(0,15,0,13,0,35\right) & 3 & 3 & 3 & 0 & 8 & 11 & 11 & 6 & 12 & 13 & 13 & 2 \\
\mathfrak{g}_{6.N17} & 25 & 14 & \left(0,0,0,12,14,15+23\right) & 13+26-45 & \left(0,0,0,15+32,16,13\right) & 3 & 3 & 3 & 0 & 5 & 6 & 6 & 2 & 6 & 7 & 7 & 2 \\
\mathfrak{g}_{6.N15} & 24 & 15 & \left(0,0,0,12,14,15+23+24\right) & 13+26-45 & \left(0,0,0,15+32+36,16,13\right) & 3 & 3 & 3 & 0 & 5 & 6 & 6 & 2 & 6 & 7 & 7 & 2 \\
\mathfrak{g}_{5.6}\oplus\mathfrak{g}_{1} & 17 & 16 & \left(0,0,0,12,14,15+24\right) & 13+26-45 & \left(0,0,0,15+36,16,13\right) & 3 & 3 & 3 & 0 & 5 & 6 & 6 & 2 & 6 & 7 & 7 & 2 \\
\mathfrak{g}_{5.2}\oplus\mathfrak{g}_{1} & 15 & 17 & \left(0,0,0,12,14,15\right) & 13+26-45 & \left(0,0,0,15,16,13\right) & 3 & 3 & 3 & 0 & 5 & 6 & 6 & 2 & 6 & 7 & 7 & 2 \\
\mathfrak{g}_{6.N9} & 19 & 19 & \left(0,0,0,12,14,13+42\right) & 15+26+34 & \left(0,16,0,15+63,0,13\right) & 3 & 3 & 3 & 0 & 5 & 8 & 8 & 6 & 6 & 10 & 10 & 8 \\
\mathfrak{g}_{6.N8} & 18 & 20 & \left(0,0,0,12,14,23+24\right) & 16-34+25 & \left(0,54+53,15,0,0,13\right) & 3 & 3 & 3 & 0 & 5 & 8 & 8 & 6 & 6 & 10 & 10 & 8 \\
\mathfrak{g}_{6.N16} & 26 & 23 & \left(0,0,0,12,14-23,15+34\right) & 16+35+24 & \left(0,14+36,0,16+35,0,15\right) & 3 & 3 & 3 & 0 & 4 & 7 & 7 & 6 & 4 & 7 & 7 & 6 \\
\mathfrak{g}_{6.N10} & 20 & 24 & \left(0,0,0,12,14+23,13+42\right) & 15+2\times26+34 & \left(0,16+\frac{1}{2}\times35,0,15+\frac{1}{2}\times63,0,\frac{1}{2}\times13\right) & 3 & 3 & 3 & 0 & 5 & 8 & 8 & 6 & 6 & 10 & 10 & 8 \\
\midrule[0.02em]
\mathfrak{g}_{4.1}\oplus2\mathfrak{g}_{1} & 9 & 26 & \left(0,0,0,0,12,15\right) & 16+25+34 & \left(0,14,0,13,0,0\right) & 4 & 4 & 4 & 0 & 7 & 9 & 9 & 4 & 8 & 12 & 12 & 8 \\
\mathfrak{g}_{5.5}\oplus\mathfrak{g}_{1} & 8 & 27 & \left(0,0,0,0,12,14+25\right) & 13+26+45 & \left(0,0,0,15+36,0,13\right) & 4 & 4 & 4 & 0  & 7 & 9 & 9 & 4 & 8 & 12 & 12 & 8 \\
\mathfrak{g}_{6N.4} & 6 & 28 & \left(0,0,0,0,12,14+23\right) & 13+26+45 & \left(0,0,0,15+32,0,13\right) & 4 & 4 & 4 & 0 & 8 & 10 & 10 & 4 & 10 & 12 & 12 & 4 \\
2\mathfrak{g}_{3.1} & 5 & 29 & \left(0,0,0,0,12,34\right) & 15+36+24 & \left(0,15,0,36,0,0\right) & 4 & 4 & 4 & 0 & 8 & 10 & 10 & 4 & 10 & 11 & 11 & 2 \\
\mathfrak{g}_{5.1}\oplus\mathfrak{g}_{1} & 4 & 30 & \left(0,0,0,0,12,13\right) & 16+25+34 & \left(0,15,0,13,0,0\right) & 4 & 4 & 4 & 0 & 9 & 11 & 11 & 4 & 12 & 14 & 14 & 4 \\
\mathfrak{g}_{6N.5} & 7 & 31 & \left(0,0,0,0,13+42,14+23\right) & 16+25+34 & \left(0,16+35,0,15+63,0,0\right) & 4 & 4 & 4 & 0 & 8 & 10 & 10 & 4 & 10 & 11 & 11 & 2 \\
\midrule[0.02em]
\mathfrak{g}_{3.1}\oplus3\mathfrak{g}_{1} & 2 & 33 & \left(0,0,0,0,0,12\right) & 16+23+45 & \left(0,13,0,0,0,0\right) & 5 & 5 & 5 & 0 & 11 & 12 & 12 & 2 & 14 & 16 & 16 & 4 \\
\midrule[0.02em]
6\mathfrak{g}_1 & 1& 34 & \left(0,0,0,0,0,0\right) & 12+34+56 & \left(0,0,0,0,0,0\right) & 6 & 6 & 6 & 0 & 15 & 15 & 15 & 0 & 20 & 20 & 20 & 0 \\
\bottomrule
\end{tabular}
}
\caption{Summary of the dimensions of the symplectic Bott-Chern and Aeppli cohomologies for $6$-dimensional nilmanifolds.}
\label{table:dim-cohom-6-nilmfd}
\end{sidewaystable}

\section{Twisted symplectic cohomologies and twisted Hard Lefschetz Condition}\label{sec:twisted-sympl-cohom}

In this section, we study twisted symplectic cohomologies, and in particular twisted Hard Lefschetz Condition and $D_{\phi}D_{\phi}^{\Lambda}$-Lemma.

\medskip

Let $X$ be be a $2n$-dimensional compact manifold endowed with a symplectic structure $\omega$. 
For a real or complex vector space $V$, consider a trivial vector bundle $E_{\phi}=X\times V$ with a connection form $\phi\in \wedge^{1}X\times \End(V)$.

\subsection{Twisted symplectic cohomologies}
Define the operator
$$ D_{\phi} \;:=\; \de+\phi $$
on the space $\wedge^{\bullet}\left(X;E_{\phi}\right)$ of the differential forms with values in the vector bundle $E_{\phi}$.

Define the operators
\[\ast_{\omega} \colon \wedge^{\bullet}\left(X;E_{\phi}\right)\to \wedge^{2n-\bullet}\left(X;E_{\phi}\right) \;,
\]
and
\begin{eqnarray*}
 L \colon \wedge^{\bullet}\left(X;E_{\phi}\right)\to \wedge^{\bullet+2}\left(X;E_{\phi}\right) \;, \\[5pt]
 \Lambda \colon \wedge^{\bullet}\left(X;E_{\phi}\right)\to \wedge^{\bullet-2}\left(X;E_{\phi}\right) \;, \\[5pt]
 H \colon \wedge^{\bullet}\left(X;E_{\phi}\right)\to \wedge^{\bullet}\left(X;E_{\phi}\right) \;,
\end{eqnarray*}
as the natural extensions of the operators in Section \ref{ope}.

Define
 $$ D_{\phi}^{\Lambda} \;:=\; \left[D_{\phi},\Lambda \right] \;. $$
By the same way as in \cite[Theorem 2.2.1]{brylinski}, (see  also \cite[Proposition 5.1]{cavalcanti},) on $\wedge^{k}\left(X;E_{\phi}\right)$,
we have
$$ D_{\phi}^{\Lambda}\lfloor_{\wedge^{k}\left(X;E_{\phi}\right)} \;=\; (-1)^{k+1} \star_{\omega} D_{\phi}\star_{\omega} \;. $$

\medskip

We suppose now that $\phi$ is flat, i.e., $\de\phi+\phi\wedge\phi=0$.
Then the pair $(\wedge^{\bullet}\left(X;E_{\phi}\right), D_{\phi})$ is a differential graded module over the differential graded algebra $\wedge^{\bullet}X$.
Moreover we have
$$ \left(D_{\phi}^{\Lambda}\right)^{2} \;=\; 0 \;. $$

Now we define the {\em twisted symplectic cohomologies}
$$ H^\bullet_{dR}(X; E_{\phi}) \;:=\; \frac{\ker D_{\phi}}{\imm D_{\phi}} \qquad \text{ and } \qquad H^\bullet_{D_{\phi}^\Lambda}(X;E_{\phi}) \;:=\; \frac{\ker D_{\phi}^\Lambda}{\imm D_{\phi}^\Lambda} \;, $$
and
$$ H^\bullet_{BC}(X;E_{\phi}) \;:=\; \frac{\ker D_{\phi}\cap\ker D_{\phi}^\Lambda}{\imm D_{\phi}D_{\phi}^\Lambda} \qquad \text{ and } \qquad H^\bullet_{A}(X;E_{\phi}) \;:=\; \frac{\ker D_{\phi}D_{\phi}^\Lambda}{\imm D_{\phi}+\imm D_{\phi}^\Lambda} \;. $$

\medskip

By the same way as in \cite[Section 1]{yan}, (see  also \cite[Proposition 5.2]{cavalcanti},) we have the relations

\begin{eqnarray*}
 \left[ D_{\phi},L \right] \;=\; 0\;, & \qquad \left[ D_{\phi}^{\Lambda}, L \right] \;=\; -D_{\phi} \;, \qquad & L \;=\; -\star_{\omega}\Lambda\star_{\omega} \;, \\[5pt]
 \left[ D_{\phi}^{\Lambda}, \Lambda \right] \;=\; 0 \;, & \qquad \left[ D_{\phi}, \Lambda \right] \;=\; D_{\phi}^{\Lambda} \;, \qquad & \Lambda \;=\; -\star_{\omega}L\star_{\omega} \\[5pt]
 \left[ \Lambda,L \right] \;=\; H \;, & \qquad \left[ L,H \right] \;=\; 2L \;, \qquad & \left[ \Lambda, H \right] \;=\; -2\Lambda \;.
\end{eqnarray*}

By $D_{\phi}^{\Lambda}=(-1)^{k+1} \ast_{\omega} D_{\phi}\star_{\omega}$, we get that the operator $\star_{\omega}\colon \wedge^{\bullet}\left(X;E_{\phi}\right)\to \wedge^{2n-\bullet}\left(X;E_{\phi}\right)$ induces the isomorphism
$$ H^\bullet_{dR}(X; E_{\phi}) \stackrel{\simeq}{\to} H^{2n-\bullet}_{D_{\phi}^\Lambda}(X;E_{\phi}).
$$

Take an almost-complex structure $J$ being compatible with $\omega$, that is, $g:=\omega(\sspace,\, J\ssspace)$ is a $J$-Hermitian metric with associated fundamental form $\omega$. Recall that there is a canonical way to construct $J$, see, e.g., \cite[Proposition 12.6]{cannasdasilva}. Consider the Hodge-$*$-operator $*_g\colon \wedge^{\bullet}\left(X;E_{\phi}\right) \to \wedge^{2n-\bullet}\left(X;E_{\phi}^{\ast}\right)$ associated to $g$ where $E_{\phi}^{\ast}$ is the dual of $E_{\phi}$.
Consider the $4^{\text{th}}$-order self-adjoint differential operators
\begin{eqnarray*}
 D_{D_{\phi}+D_{\phi}^\Lambda} &:=& \left(D_{\phi}D_{\phi}^\Lambda\right)\left(D_{\phi}D_{\phi}^\Lambda\right)^* + \left(D_{\phi}D_{\phi}^\Lambda\right)^*\left(D_{\phi}D_{\phi}^\Lambda\right) + \left(D_{\phi}^*D_{\phi}^\Lambda\right)\left(D_{\phi}^*D_{\phi}^\Lambda\right)^* + \left(D_{\phi}^*D_{\phi}^\Lambda\right)^*\left(D_{\phi}^*D_{\phi}^\Lambda\right) \\[5pt]
 && + D_{\phi}^*D_{\phi}+\left(D_{\phi}^\Lambda\right)^*D_{\phi}^\Lambda
\end{eqnarray*}
and
\begin{eqnarray*}
 D_{D_{\phi}D_{\phi}^\Lambda} &:=& \left(D_{\phi}D_{\phi}^\Lambda\right)\left(D_{\phi}D_{\phi}^\Lambda\right)^* + \left(D_{\phi}D_{\phi}^\Lambda\right)^*\left(D_{\phi}D_{\phi}^\Lambda\right) + \left(D_{\phi}\left(D_{\phi}^\Lambda\right)^*\right)\left(D_{\phi}\left(D_{\phi}^\Lambda\right)^*\right)^* \\[5pt]
 && + \left(D_{\phi}\left(D_{\phi}^\Lambda\right)^*\right)^*\left(D_{\phi}\left(D_{\phi}^\Lambda\right)^*\right) + D_{\phi}D_{\phi}^* + D_{\phi}^\Lambda\left(D_{\phi}^\Lambda\right)^* \;.
 \end{eqnarray*}
Then, as similar to \cite{tseng-yau-1}, these operators are elliptic, and hence induce  orthogonal decompositions
 $$ \wedge^\bullet \left(X;E_{\phi}\right) \;=\; \ker D_{D_{\phi}+D_{\phi}^\Lambda} \stackrel{\perp}{\oplus} D_{\phi}D_{\phi}^\Lambda \wedge^\bullet \left(X;E_{\phi}\right) \stackrel{\perp}{\oplus} \left(D_{\phi}^*\wedge^{\bullet+1}\left(X;E_{\phi}\right) + \left(D_{\phi}^\Lambda\right)^*\wedge^{\bullet-1}\left(X;E_{\phi}\right)\right) \; $$
and
$$\wedge^\bullet \left(X;E_{\phi}\right) \;=\; \ker D_{D_{\phi}D_{\phi}^\Lambda} \stackrel{\perp}{\oplus} \left(D_{\phi}\wedge^{\bullet-1}\left(X;E_{\phi}\right) + D_{\phi}^\Lambda\wedge^{\bullet+1}\left(X;E_{\phi}\right)\right) \stackrel{\perp}{\oplus} \left(D_{\phi}D_{\phi}^\Lambda\right)^* \wedge^\bullet \left(X;E_{\phi}\right)\;.
$$
Therefore we have the isomorphisms
 $$ H^{\bullet}_{BC}(X;E_\phi) \;\simeq\; \ker D_{D_{\phi}+D_{\phi}^\Lambda} \qquad \text{ and } \qquad H^{\bullet}_{A}(X;E_\phi) \;\simeq\; \ker D_{D_{\phi}D_{\phi}^\Lambda}\;. $$
By  noting that $*_g D_{D_{\phi}+D_{\phi}^\Lambda}=D_{D_{\phi}D_{\phi}^\Lambda}*_g$, we have the isomorphism
$$ *_g\colon H^{\bullet}_{BC}(X;E_{\phi})\stackrel{\simeq}{\to}H^{2n-\bullet}_{A}(X;E_{\phi}^{\ast}) \;. $$

\subsection{Twisted Hard Lefschetz Condition}

As in \cite[page 102]{brylinski}, we define $c\in \wedge^{\bullet}\left(X;E_{\phi}\right)$ to be {\em symplectically-harmonic} if $D_{\phi}c=D_{\phi}^{\Lambda}c=0$.

We say that $\left(X,\, \omega\right)$ satisfies the {\em $E_{\phi}$-twisted Hard Lefschetz Condition} if, for each $1\le k\le n$, the linear map $[\omega^{k}]\cp \sspace \colon H^{n-k}_{dR}(X; E_{\phi})\to H^{n+k}_{dR}(X; E_{\phi})$ is an isomorphism.

\medskip

By the above relations, by using the $\slf_{2}(\C)$-representation theory, we have the following result as an application of Theorem \ref{har}.

\begin{thm}
Let $X$ be a $2n$-dimensional compact manifold endowed with a symplectic structure $\omega$.
Let $E_{\phi}=X\times V$ be a trivial vector bundle on $X$ with a connection form $\phi\in \wedge^{1}X\times \End(V)$. Suppose that $\phi$ is flat.

The following two conditions are  equivalent:
\begin{enumerate}
\item for each $1\le k\le n$, the linear map $[\omega^{k}]\cp \sspace \colon H^{n-k}_{dR}(X; E_{\phi})\to H^{n+k}_{dR}(X; E_{\phi})$ is surjective;
\item there is a symplectically-harmonic representative in each cohomology class in $H^{\bullet}_{dR}(X;E_{\phi})$.
\end{enumerate}
\end{thm}

By the Poincar\'e duality for local systems and Corollary \ref{sel}, we have the following result.
\begin{thm}
Let $X$ be a $2n$-dimensional compact manifold endowed with a symplectic structure $\omega$.
Let $E_{\phi}=X\times V$ be a trivial vector bundle on $X$ with a connection form $\phi\in \wedge^{1}X\times \End(V)$. Suppose that $\phi$ is flat, and that the dual flat bundle of $E_{\phi}$ is isomorphic to $E_{\phi}$ itself.

Then the following two conditions are equivalent:
\begin{enumerate}
\item $(X,\omega)$ satisfies the $E_{\phi}$-twisted Hard Lefschetz Condition;
\item there is a symplectically-harmonic representative in each cohomology class in $H^{\bullet}\left(X;E_{\phi}\right)$.
\end{enumerate}
\end{thm}

\subsection{$D_{\phi}D_{\phi}^{\Lambda}$-Lemma}

We say that $\left(X,\, \omega\right)$ satisfies the {\em $D_{\phi}D_{\phi}^{\Lambda}$-Lemma} if we have
$$ \imm D_{\phi}^{\Lambda}\cap\ker D_{\phi} \;=\; \imm D_{\phi}D_{\phi}^{\Lambda} \;=\; \imm D_{\phi}\cap\ker D_{\phi}^{\Lambda} \;, $$
namely, the natural maps $H^\bullet_{BC}(X;E_\phi) \to H^\bullet_{dR}(X;E_\phi)$ and $H^\bullet_{BC}(X;E_\phi) \to H^\bullet_{\de^\Lambda}(X;E_\phi)$ induced by the identity are isomorphisms.

\medskip

In studying the relations between the twisted Hard Lefschetz Condition and the $D_{\phi}D_{\phi}^{\Lambda}$-Lemma, we need the following result.

\begin{prop}
Let $X$ be a $2n$-dimensional compact manifold endowed with a symplectic structure $\omega$.
Let $E_{\phi}=X\times V$ be a trivial vector bundle on $X$ with a connection form $\phi\in \wedge^{1}X\times \End(V)$. Suppose that $\phi$ is flat.

Suppose that
\begin{itemize}
\item either $E_{\phi}$ is isomorphic to trivial flat bundle $E_{0}$,
\item or $H^{0}_{dR}\left(X; E_{\phi}\right)=\{0\}$.
\end{itemize}
Suppose also that $\left(X,\, \omega\right)$ satisfies the $E_{\phi}$-twisted Hard Lefschetz Condition.

Then we have
$$ \imm D_{\phi}^{\Lambda}\cap\ker D_{\phi} \;=\; \imm D_{\phi}\cap\ker D_{\phi}^{\Lambda} \;. $$
\end{prop}

\begin{proof}
In the case $E_{\phi}\simeq E_{0}$, we have $\left( \wedge^{\bullet}\left(X;E_{\phi}\right), \, D_{\phi} \right) \simeq \left( \wedge^{\bullet}(X)\otimes \C^{n} ,\, \de \right)$.
Hence we can prove as in the ordinary case, see \cite[Proposition 5.4]{cavalcanti}.

\medskip

Suppose $H^{0}_{dR}\left(X; E_{\phi}\right)=\{0\}$.
Let $\alpha\in \wedge^{1}\left(X;E_{\phi}\right)$ with $D_{\phi}D^{\Lambda}_{\phi}\alpha=0$.
Then, by the assumption, we have $D^{\Lambda}_{\phi}\alpha=0$.
Hence by Proposition \ref{del12}, we have
$$ \imm D_{\phi}^{\Lambda}\cap \ker D_{\phi} \;=\; \imm D_{\phi}\cap \imm D_{\phi}^{\Lambda} \;. $$
By this equation and the relation $D_{\phi}^{\Lambda}\lfloor_{\wedge^{k}\left(X;E_{\phi}\right)} = (-1)^{k+1} \star_{\omega} D_{\phi}\star_{\omega}$, we have also
$$ \imm D_{\phi}\cap \ker D^{\Lambda}_{\phi} \;=\; \imm D_{\phi}\cap \imm D_{\phi}^{\Lambda} \;. $$
Hence the proposition follows.
\end{proof}

By this proposition and Proposition \ref{ddlam} and Proposition \ref{delin}, we have the following result.

\begin{cor}
Let $X$ be a $2n$-dimensional compact manifold endowed with a symplectic structure $\omega$.
Let $E_{\phi}=X\times V$ be a trivial vector bundle on $X$ with a connection form $\phi\in \wedge^{1}X\times \End(V)$. Suppose that $\phi$ is flat.
Suppose that the monodromy representation of $E_{\phi}$ is semi-simple.

Consider the following two conditions:
\begin{enumerate}
\item $(X,\omega)$ satisfies the $E_{\phi}$-twisted Hard Lefschetz Condition;
\item  $(X,\omega)$ satisfies the $D_{\phi}D_{\phi}^{\Lambda}$-Lemma.
\end{enumerate}

Then the first condition implies the second one.
Moreover if the dual flat bundle of $E_{\phi}$ is isomorphic to $E_{\phi}$ itself, then the two conditions are equivalent. 
\end{cor}

Arguing as in \cite[Proposition 3.13]{tseng-yau-1}, one has the following result.

\begin{cor}
Let $X$ be a $2n$-dimensional compact manifold endowed with a symplectic structure $\omega$.
Let $E_{\phi}=X\times V$ be a trivial vector bundle on $X$ with a connection form $\phi\in \wedge^{1}X\times \End(V)$. Suppose that $\phi$ is flat.
Suppose that the monodromy representation of $E_{\phi}$ is semi-simple.

Consider the following two conditions :
\begin{enumerate}
\item $(X,\omega)$ satisfies the $E_{\phi}$-twisted Hard Lefschetz Condition;
\item the map $H_{BC}^{\bullet}\left(X;E_{\phi}\right)\to H_{dR}^{\bullet}\left(X;E_{\phi}\right)$ is an isomorphism.
\end{enumerate}

Then the first condition imply the second one.
Moreover if the dual flat bundle of $E_{\phi}$ is isomorphic to $E_{\phi}$ itself, then the two conditions are equivalent. 
\end{cor}

The following result is a straightforward corollary.

\begin{cor}
Let $X$ be a $2n$-dimensional compact manifold endowed with a symplectic structure $\omega$.
Let $E_{\phi}=X\times V$ be a trivial vector bundle on $X$ with a connection form $\phi\in \wedge^{1}X\times \End(V)$. Suppose that $\phi$ is flat.
Suppose that the monodromy representation of $E_{\phi}$ is semi-simple.
We assume that $(X,\omega)$ satisfies the $E_{\phi}$-twisted Hard Lefschetz Condition.

Then for each $k\in\Z$, we have $\dim_\R H_{BC}^{k}\left(X;E_{\phi}\right)=\dim_\R H_{dR}^{k}\left(X;E_{\phi}\right)$.
\end{cor}

 In \cite{Sim}, Simpson showed the following result.

\begin{thm}[{\cite[Lemma 2.6]{Sim}}]
Let $(X,\omega)$ be a compact K\"ahler manifold and $E_{\phi}$ a flat bundle over $X$ whose monodromy representation is semi-simple.
Then  $(X,\omega)$ satisfies the $E_{\phi}$-twisted Hard Lefschetz Condition.
\end{thm}

\section{Twisted cohomologies on solvmanifolds}\label{sec:twisted-sympl-cohom-solvmfd}

In this section, we study twisted symplectic cohomologies for special solvmanifolds.

\medskip

Let $G$ be a connected simply-connected solvable Lie group. Denote by $\g$ its associated Lie algebra, and by $\rho\colon G\to \GL(V_{\rho})$ a representation on a real or complex vector space $V_{\rho}$.

We consider the cochain complex $\wedge^{\bullet} \g^{\ast}$ with the  derivation $\de$ which is the dual to the Lie bracket of $\g$.
Then the pair
$$ \left( \wedge^{\bullet} \g^{\ast}\otimes V_{\rho},\, \de_{\rho}:=\de+\rho_{\ast} \right) $$
is a differential graded module  over the differential graded algebra $\wedge^{\bullet} \g^{\ast}$.
Here $\rho_{\ast} \in \g^{\ast}\otimes \mathfrak{gl}(V_{\rho})$ is the derivation of $\rho$.
We can consider the cochain complex $\left( \wedge^{\bullet} \g^{\ast}\otimes V_{\rho}, \, \de_{\rho} \right)$ given by the twisted $G$-invariant differential forms on $G$.

\medskip

Suppose that $G$ has a lattice $\Gamma$.
Since $\pi_{1}(\solvmfd)=\Gamma$, we have a flat vector bundle $E_{\rho_{\ast}}$ with flat connection $D_{\rho_{\ast}}$ on $\solvmfd$ whose monodromy is $\rho\lfloor_{\Gamma}$.
We can regard $E_{\rho_{\ast}}$ as the  flat bundle $\solvmfd\times V_{\rho}$ with the  connection form $\rho_{\ast}$, and we have the inclusion
\[ \iota\colon \wedge^{\bullet}\g^{\ast}\otimes V_{\rho} \hookrightarrow \wedge^{\bullet}\left( \solvmfd; E_{\rho_{\ast}} \right) \]
of cochain complexes.
Consider the natural extension
$$ \mu\colon \wedge^{\bullet}\left( \solvmfd; E_{\rho_{\ast}} \right)\to \wedge^{\bullet}\g^{\ast}\otimes V_{\rho} $$
of the F.~A. Belgun symmetrization map, \cite[Theorem 7]{belgun}.
Then this map also satisfies
$$ D_{\rho_{\ast}} \circ \mu \;=\; \mu\circ D_{\rho_{\ast}} \qquad \text{ and } \qquad \mu\circ \iota \;=\; \id \;. $$

\subsection{Twisted symplectic cohomologies of special solvmanifolds}

In some special cases, the inclusion $\iota\colon \wedge^{\bullet}\g^{\ast}\otimes V_{\rho} \hookrightarrow \wedge^{\bullet}\left( \solvmfd; E_{\rho_{\ast}} \right)$ is a quasi-isomorphism.

\begin{thm}[{\cite[Corollary 4.2]{hattori}, \cite[Theorem 8.2, Corollary 8.1]{mostow}}]\label{lietw}
Let $G$ be a connected simply-connected solvable Lie group with a lattice $\Gamma$.

If:
\begin{description}
 \item[(H)] either the representation $\rho\oplus\Ad$ is $\R$-triangular, \cite{hattori},
 \item[(M)] or the two images $(\rho\oplus\Ad)(G)$ and $(\rho\oplus\Ad)(\Gamma)$ have the same Zariski-closure in $GL(V_{\rho})\times \Aut(\g_{\C})$,
\end{description}
then the inclusion $\iota \colon \wedge^{\bullet}\g^{\ast}_{\C}\otimes V_{\rho} \hookrightarrow \wedge^{\bullet}\left( \solvmfd; E_{\rho_{\ast}} \right)$ induces the isomorphism
$$ H^{\bullet}_{dR}(\g; V_{\rho}) \stackrel{\simeq}{\to} H^{\bullet}_{dR}(\solvmfd; E_{\rho_{\ast}}) \;. $$
\end{thm}

\medskip

We suppose that $\solvmfd$ admits a $G$-left-invariant symplectic structure $\omega\in \wedge^{2}\g^{\ast}$.
Then the operators $L$, $\Lambda$, $\star_{\omega}$, and $D_{\phi}^{\Lambda}$ are defined on $\wedge^{\bullet}\g^{\ast}\otimes V_{\rho}$.
We consider the cohomologies $H^{\bullet}_{D_{\rho}}\left( \g; V_{\rho} \right)$, $H^{\bullet}_{D_{\rho}^{\Lambda}}\left( \g; V_{\rho} \right)$
and $H^{\bullet}_{BC}\left( \g; V_{\rho} \right)$ of $\wedge^{\bullet}\g^{\ast}_{\C}\otimes V_{\rho}$.

We have the following result.

\begin{thm}\label{thm:twisted-sympl-bc-cohom-solvmfd}
Let $\solvmfd$ be a solvmanifold endowed with a $G$-left-invariant symplectic structure $\omega\in \wedge^{2}\g^{\ast}$.
We suppose that either condition {\itshape (H)} or condition {\itshape (M)} in Theorem \ref{lietw} holds.

Then the inclusion $\iota\colon \wedge^{\bullet}\g^{\ast}_{\C}\otimes V_{\rho}\subset \wedge^{\bullet}\left( \solvmfd; E_{\rho_{\ast}} \right)$ induces the isomorphism
$$ H^{\bullet}_{BC}\left( \g; V_{\rho} \right) \stackrel{\simeq}{\to} H^{\bullet}_{BC}\left( \solvmfd; E_{\rho_{\ast}} \right) \;. $$
\end{thm}

\begin{proof}
By Theorem \ref{lietw}, we have that $\iota$ induces the isomorphism
$H^{\bullet}_{D_{\rho}}\left( \g; V_{\rho} \right) \stackrel{\simeq}{\to} H^{\bullet}_{D_{\rho}}\left( \solvmfd; E_{\rho_{\ast}} \right)$.
By using the symplectic-$\star$-operator $\star_{\omega} \colon \wedge^{\bullet}\g^\ast\otimes V_\rho \to \wedge^{\dim X-\bullet}\g^\ast\otimes V_\rho$, we have that $\iota$ induces the isomorphism
$H^{\bullet}_{D^{\Lambda}_{\rho}}\left( \g; V_{\rho} \right) \stackrel{\simeq}{\to} H^{\bullet}_{D^{\Lambda}_{\rho}}\left( \solvmfd; E_{\rho_{\ast}} \right)$ as in the proof of Theorem \ref{thm:sympl-cohom-compl-solv}.

Consider the F.~A. Belgun symmetrization map $\mu\colon \wedge^{\bullet}\left( \solvmfd; E_{\rho_{\ast}} \right)\to \wedge^{\bullet}\g^{\ast}\otimes V_{\rho}$ as above.
Then, since $\omega$ is $G$-left-invariant, $\mu$ commutes with the operators $L$, $\Lambda$, $\star_{\omega}$, and $D_{\phi}^{\Lambda}$.
Hence we get that $\iota$ induces the isomorphism
$$H^{\bullet}_{BC}\left( \g; V_{\rho} \right) \stackrel{\simeq}{\to} H^{\bullet}_{BC}\left( \solvmfd; E_{\rho_{\ast}} \right)$$
from Corollary \ref{ives}.
\end{proof}

\begin{cor}
Let $\solvmfd$ be a solvmanifold endowed with a $G$-left-invariant symplectic structure $\omega\in \wedge^{2}\g^{\ast}$.
We suppose that either condition {\itshape (H)} or condition {\itshape (M)} in Theorem \ref{lietw} holds.

Then the inclusion $\iota\colon \wedge^{\bullet}\g^{\ast}_{\C}\otimes V_{\rho}\subset \wedge^{\bullet}\left( \solvmfd; E_{\rho_{\ast}} \right)$ induces the isomorphism
$$ H^{\bullet}_{A}\left( \g; V_{\rho} \right) \stackrel{\simeq}{\to} H^{\bullet}_{A}\left( \solvmfd; E_{\rho_{\ast}} \right) \;. $$
\end{cor}
\begin{proof}
By the F.~A. Belgun symmetrization map $\mu\colon \wedge^{\bullet}\left( \solvmfd; E_{\rho_{\ast}} \right)\to \wedge^{\bullet}\g^{\ast}\otimes V_{\rho}$ as above,
 the induced map $\iota\colon H^{\bullet}_{A}\left( \g; V_{\rho} \right) \to H^{\bullet}_{A}\left( \solvmfd; E_{\rho_{\ast}} \right)$ is injective.
Hence it is sufficient to show that there exists an isomorphism $H^\bullet_{A}(\g;V_{\rho}) \stackrel{\simeq}{\to} H^\bullet_{A}(\solvmfd; E_{\rho_{\ast}})$. 

 Let $J$ be a $G$-left-invariant $\omega$-compatible almost-complex structure on $X$, (see, e.g., \cite[Proposition 12.6]{cannasdasilva},) and consider the $G$-left-invariant $J$-Hermitian metric $g:=\omega\left(\sspace, \, J\ssspace\right)$.
Consider the Hodge-$*$-operator  $*_g\colon \wedge^{\bullet}\g^{\ast}_{\C}\otimes V_{\rho} \to \wedge^{2n-\bullet}\g^{\ast}_{\C}\otimes V_{\check\rho}$ on left-invariant forms where $\check\rho$ is the dual representation of $\rho$.
Like the duality between Bott-Chern and Aeppli cohomologies of compact symplectic manifolds, we have
 the isomorphism $ H^\bullet_{A}(\g;V_{\rho}) \stackrel{\simeq}{\to} H^{2n-\bullet}_{BC}(\g;V_{\check\rho})$ induced by  $*_g$.
If $\rho$ satisfies  either condition {\itshape (H)} or condition {\itshape (M)} in Theorem \ref{lietw}, the dual representation $\check\rho$ also does.
Hence by Theorem \ref{thm:twisted-sympl-bc-cohom-solvmfd}, we have the isomorphism
$$H^{\bullet}_{BC}\left( \g; V_{\check\rho} \right) \stackrel{\simeq}{\to} H^{\bullet}_{BC}\left( \solvmfd; E_{\rho_{\ast}}^{\ast} \right) \;.$$
By the duality between Bott-Chern and Aeppli cohomologies on $\solvmfd$,  we have $ H^\bullet_{A}(\solvmfd; E_{\rho_{\ast}}) \stackrel{\simeq}{\to} H^{2n-\bullet}_{BC}(\solvmfd; E_{\rho_{\ast}}^{\ast})$ and hence we have $H^\bullet_{A}(\g;V_{\rho}) \stackrel{\simeq}{\to} H^\bullet_{A}(\solvmfd; E_{\rho_{\ast}})$.
\end{proof}

\subsection{Twisted minimal model of solvmanifolds}
Consider a connected simply-connected solvable Lie group $G$ with a lattice $\Gamma$.
Consider the adjoint action $\ad \colon \g \ni X \mapsto \ad_X := \left[X,\sspace\right] \in \Der(\g)$, and, for any $X\in\g$, consider its unique Jordan decomposition $\ad_X = \left(\ad_X\right)_{\mathrm{s}}+\left(\ad_X\right)_{\mathrm{n}}$, where $\left(\ad_X\right)_{\mathrm{s}}\in\gl(\g)$ is semi-simple and $\left(\ad_X\right)_{\mathrm{n}}\in\gl(\g)$ is nilpotent, see, e.g., \cite[I\hspace{-.1em}I.1.10]{DER}.

Denote by $\n$ the nilradical of $\g$, and let $V$ be an $\R$-vector sub-space of $\g$ such that
\begin{inparaenum}[\itshape (i)]
\item $\g=V\oplus \n$ in the category of $\R$-vector spaces, and,
\item for any  $A,B\in V$, it holds that $\left(\ad_A\right)_{\mathrm{s}}(B)=0$,
\end{inparaenum}
see, e.g., \cite[Proposition I\hspace{-.1em}I\hspace{-.1em}I.1.1] {DER}.
Hence, define the map $\ad_{\mathrm{s}} \colon \g = V \oplus \n \ni \left(A,X\right) \mapsto \left(\ad_{\mathrm{s}}\right)_{A+X} := (\ad_{A})_{\mathrm{s}} \in \Der(\g)$.
Moreover, one has that
\begin{inparaenum}[\itshape (i)]\setcounter{enumi}{2}
\item $\left[\ad_{\mathrm{s}}(\g), \ad_{\mathrm{s}}(\g)\right]=\{0\}$, and
\item $\ad_{\mathrm{s}} \colon \g \to \gl(\g)$ is $\R$-linear,
\end{inparaenum}
see, e.g., \cite[Proposition {{I\hspace{-.1em}I\hspace{-.1em}I}.1.1}] {DER}.

The map $\ad_{\mathrm{s}}\colon \g\to \gl(\g)$ is actually a representation of $\g$ such that its image $\ad_{\mathrm{s}}(\g)$ is Abelian and consists of semi-simple elements. Hence denote by $\Ad_{\mathrm{s}}\colon G\to {\Aut}(\g)$ the unique representation which lifts ${\ad}_{\mathrm{s}}\colon \g \to \gl(\g)$, see, e.g., \cite[Theorem 3.27]{warner}, and by $\Ad_{\mathrm{s}}\colon G\to {\Aut}\left(\g_\C\right)$ its natural $\C$-linear extension.

 Let $T$ be the Zariski-closure of $\Ad_{s}(G)$ in $\Aut(\g_{\C})$.
Let 
\[{\mathcal C} \;:=\; \left\{ \beta\circ \Ad_{\mathrm{s}} \in \Hom\left(G;\C^{\ast}\right) \st \beta\in  \Char(T) \right\} \;. \]
For $\alpha\in\mathcal{C}$, consider $\alpha\colon G \to \GL(V_\alpha)\simeq\C^\ast$.
We consider the differential graded algebra
$$\bigoplus_{\alpha\in {\mathcal C}}\wedge^{\bullet} \g^\ast \otimes V_{\alpha} 
$$
with the $T$-action.
Denote by
$$ \left(\bigoplus_{\alpha\in {\mathcal C}}\wedge^{\bullet} \g^\ast \otimes V_{\alpha} \right)^{T} $$
the sub-differential graded algebra which consists of $T$-invariant elements.

Since ${\Ad}_{\mathrm{s}}(G)\subseteq\Aut(\g_\C)$ consists of simultaneously diagonalizable elements, let $\left\{ X_{1},\ldots ,X_{n} \right\}$ be a basis of $\g_{\C}$ with respect to which $\Ad_{\mathrm{s}} = \diag\left(\alpha_1, \ldots, \alpha_n\right) \colon G \to \Aut(\g_\C)$ for some characters $\alpha_1 \in \Hom(G;\C^*)$, \dots, $\alpha_n \in \Hom(G;\C^*)$, and let $\left\{ x_{1},\ldots,x_{n}\right\}$ be its dual basis of $\g_\C^\ast$.
Then we have
$$ \left(\bigoplus_{\alpha\in {\mathcal C}}\wedge^{\bullet} \g^\ast\otimes V_{\alpha} \right)^{T} \;=\; \wedge^{\bullet}\left\langle x_{1}\otimes v_{\alpha_{1}},\dots,x_{n}\otimes v_{\alpha_{n}}\right\rangle $$
where $\left\{ v_{\alpha_{j}} \right\}$ is a basis of $V_{\alpha_{j}}$ for each $j\in\{1,\ldots,n\}$.
In \cite[Section 5]{kasuya-jdg}, the second author showed that we have a differential graded algebra isomorphism
\[ \wedge^{\bullet}\left\langle x_{1}\otimes v_{\alpha_{1}},\dots,x_{n}\otimes v_{\alpha_{n}}\right\rangle \;\simeq\; \wedge^{\bullet}\mathfrak{u}^{\ast} \]
where $\mathfrak{u}$ is the Lie algebra of the unipotent hull of $G$, which is the unipotent algebraic group determined by $G$.
Since $\mathfrak{u}$ is nilpotent, $\wedge^{\bullet}\mathfrak{u}^{\ast}$ is a minimal differential graded algebra.

We also consider the Zariski-closure $S$ of $\Ad_{s}(\Gamma)$ in $\Aut(\g_{\C})$.
For $\beta\in \Char(S)$, we denote by $E_{\beta}$ the rank one flat bundle with the monodromy representation $\beta \circ \Ad_{s}\lfloor_\Gamma $.
Consider the differential algebra
$$ \bigoplus_{\beta\in\Char(S)} \wedge^{\bullet}\left( \solvmfd; E_{\beta} \right) \;. $$
Differential graded algebras of this kind were considered by Hain in \cite{Hai} for rational homotopy on non-nilpotent spaces.
In \cite{kasuya-jdg}, the second author constructed the minimal model of $\bigoplus_{\beta\in\Char(S)} \wedge^{\bullet}\left(\solvmfd;E_{\beta}\right)$.
By $S\subseteq T$, for $\beta\in \Char(S)$, we have characters $\alpha\in {\mathcal C}$ such that $E_{\alpha_{\ast}}=E_{\beta}$.
(Since  $S\not= T$ in general,  such $\alpha$ is not unique.)
Hence we have the  map
$$\bigoplus_{\alpha\in {\mathcal C}}\wedge^{\bullet} \g^\ast\otimes V_{\alpha} \to \bigoplus_{\beta\in\Char(S)} \wedge^{\bullet}\left( \solvmfd; E_{\beta} \right).
$$
Now we consider the map $\iota\colon\wedge^{\bullet}\mathfrak{u}^{\ast}\to \bigoplus_{\beta\in\Char(S)} \wedge^{\bullet}\left(\solvmfd;E_{\beta}\right)$ given by the composition
$$\wedge^{\bullet}\mathfrak{u}^{\ast} \;\simeq\; \left(\bigoplus_{\alpha\in {\mathcal C}}\wedge^{\bullet} \g^\ast\otimes V_{\alpha} \right)^{T} \;\subseteq\; \bigoplus_{\alpha\in {\mathcal C}}\wedge^{\bullet} \g^\ast\otimes V_{\alpha} \to \bigoplus_{\beta\in\Char(S)} \wedge^{\bullet}\left(\solvmfd;E_{\beta}\right) \;.
$$

The second author proved the following result in \cite{kasuya-jdg}.

\begin{thm}[{\cite[Theorem 1.1, Theorem 5.4]{kasuya-jdg}}]\label{min}
Let $G$ be a connected simply-connected solvable Lie group with a lattice $\Gamma$.
The above map $\iota\colon\wedge^{\bullet}\mathfrak{u}^{\ast}\to \bigoplus_{\beta\in\Char(S)} \wedge^{\bullet}\left( \solvmfd; E_{\beta} \right)$ induces a cohomology isomorphism.
Hence $\wedge^{\bullet}\mathfrak{u}^{\ast}$ is the minimal model of $\bigoplus_{\beta\in\Char(S)} \wedge^{\bullet}\left( \solvmfd; E_{\beta} \right)$.
\end{thm}

\begin{note}\label{untwi-co}
We consider the maps $\iota\colon\wedge^{\bullet}\mathfrak{u}^{\ast}\to \bigoplus_{\beta\in\Char(S)} \wedge^{\bullet}\left(\solvmfd;E_{\beta}\right)$ given by the composition
$$\wedge^{\bullet}\mathfrak{u}^{\ast} \;\simeq\; \left(\bigoplus_{\alpha\in {\mathcal C}}\wedge^{\bullet} \g^\ast\otimes V_{\alpha} \right)^{T} \;\subseteq\; \bigoplus_{\alpha\in {\mathcal C}}\wedge^{\bullet} \g^\ast\otimes V_{\alpha} \to \bigoplus_{\beta\in\Char(S)} \wedge^{\bullet}\left(\solvmfd;E_{\beta}\right) \;.
$$
as above.
Let $A_{\Gamma}^{\bullet}=\iota^{-1}\left( \wedge^{\bullet}\solvmfd \otimes \C\right)$ where we consider $\wedge^{\bullet}\solvmfd \otimes \C=\wedge^{\bullet}\left( \solvmfd; E_{1_{S}} \right)$ for the trivial character $1_{S}$.
Then the  map $\iota\colon A_{\Gamma}^{\bullet}\subset \wedge^{\bullet}\solvmfd \otimes \C$ induces a cohomology isomorphism.
Set 
\[{\mathcal C}_{\Gamma} \;:=\; \left\{ \beta\circ \Ad_{\mathrm{s}} \in \Hom\left(G;\C^{\ast}\right) \st \beta\in  \Char(T),\; \left(\beta\circ \Ad_{\mathrm{s}}\right)\lfloor_\Gamma=1 \right\} \;. \]
As \cite[Corollary 7.6]{kasuya-jdg}, we have
\begin{equation}\label{eq:def-a1}A_{\Gamma}^{\bullet}=\left(\bigoplus_{\alpha\in {\mathcal C}_{\Gamma}}\wedge^{\bullet} \g^{\ast}_{\C}\otimes V_{\alpha} \right)^{T}.
\end{equation}
By using the basis $\{x_{1},\dots, x_{n}\}$ of $\g_{\C}^{\ast}$  such that
 $$\left(\bigoplus_{\alpha\in {\mathcal C}_{\Gamma}}\wedge^{\bullet} \g^{\ast}_{\C}\otimes V_{\alpha} \right)^{T}\cong \wedge^{\bullet}\left\langle x_{1}\otimes v_{\alpha_{1}},\dots,x_{n}\otimes v_{\alpha_{n}}\right\rangle$$
as above, since we have $\wedge^{\bullet} \g^{\ast}_{\C}\otimes V_{\alpha}=\alpha\cdot \wedge^{\bullet} \g^{\ast}_{\C}$ in $\wedge^{\bullet}\solvmfd \otimes \C$ for $\alpha\in {\mathcal A}_{\Gamma}$, 
  the differential graded algebra $ A^{\bullet}_\Gamma$  can be written as
\begin{equation}\label{eq:def-a}
A^{p}_\Gamma
\;=\;
\C\left\langle \alpha_{i_{1}\cdots i_{p}}\, x_{i_{1}}\wedge \dots \wedge x_{i_{p}} \;\middle\vert\; 1\le i_{1}<i_{2}<\dots <i_{p}\le n \; \text{ such that } \alpha_{i_{1}\cdots i_{p}}\lfloor_\Gamma = 1\right\rangle \;,
\end{equation}
where we have shortened $\alpha_{i_{1}\cdots i_{p}} := \alpha_{i_{1}}\cdot \cdots \cdot \alpha_{i_{p}} \in \Hom\left(G;\C^*\right)$.
\end{note}

\medskip

We suppose now that $\solvmfd$ admits a $G$-left-invariant symplectic structure $\omega\in \wedge^{2}\g$.
We assume that $\omega$ is $T$-invariant (equivalently $\omega \in A^{2}_\Gamma$).

Then the operators $L$ and $\Lambda$ on $\bigoplus_{\alpha\in {\mathcal C}}\wedge^{\bullet} \g\otimes V_{\alpha} $ commute with the $T$-action.
Hence $L$ and $\Lambda$ and the differential $D^{\Lambda}=D\Lambda-\Lambda D$ are defined on $\wedge^{\bullet}\mathfrak{u}^{\ast}\simeq \left(\bigoplus_{\alpha\in {\mathcal C}}\wedge^{\bullet} \g^\ast \otimes V_{\alpha} \right)^{T}$, where $D$ is the differential on the differential graded algebra $\wedge^{\bullet}\mathfrak{u}^{\ast}$.
Now we can regard $\omega$ as a symplectic form on the Lie algebra $\mathfrak{u}$.
The symplectic-$\star$-operator $\star_{\omega}$ is defined on  $\wedge^{\bullet}\mathfrak{u}^{\ast}$.
We consider the cohomologies $H^{\bullet}_{D}(\bigoplus_{\beta\in\Char(S)} \wedge^{\bullet}\left( \solvmfd; E_{\beta} \right))$, $H^{\bullet}_{D^{\Lambda}}(\bigoplus_{\beta\in\Char(S)} \wedge^{\bullet}\left( \solvmfd; E_{\beta} \right))$, and $H^{\bullet}_{BC}(\bigoplus_{\beta\in\Char(S)} \wedge^{\bullet}\left( \solvmfd; E_{\beta} \right))$ of $\bigoplus_{\beta\in\Char(S)} \wedge^{\bullet}\left( \solvmfd; E_{\beta} \right)$ and the cohomologies $H^{\bullet}_{D}(\mathfrak{u})$, $H^{\bullet}_{D^{\Lambda}}(\mathfrak{u})$,
and $H^{\bullet}_{BC}(\mathfrak{u})$ of $\wedge^{\bullet}\mathfrak{u}^{\ast}$.

\begin{thm}\label{sym-min}
Let $G$ be a connected simply-connected solvable Lie group with a lattice $\Gamma$ and endowed with a $G$-left-invariant $T$-invariant symplectic structure $\omega$.

The above map $\iota\colon\wedge^{\bullet}\mathfrak{u}^{\ast}\to \bigoplus_{\beta\in\Char(S)} \wedge^{\bullet}\left( \solvmfd; E_{\beta} \right)$ induces the Bott-Chern cohomology isomorphism
$$ H^{\bullet}_{BC}(\mathfrak{u}) \stackrel{\simeq}{\to} H^{\bullet}_{BC}\left(\bigoplus_{\beta\in\Char(S)} \wedge^{\bullet}\left( \solvmfd; E_{\beta} \right)\right) \;. $$
\end{thm}

\begin{proof}
Set 
$$ {\mathcal A} \;:=\; \left\{ \alpha_{i_{1}\dots i_{p}} \in \Hom(G;\C^{\ast}) \st 1\le i_{1}\le\dots\le i_{p}\le n \right\} $$
and
$$ {\mathcal A}^{\prime} \;:=\; \left\{ \beta\in\Char(S) \st \text{ there exists } \alpha\in {\mathcal A} \text{ such that } E_{\beta}=E_{\alpha_{\ast}}  \right\} \;. $$
We consider the projection
$$ p \colon \bigoplus_{\beta\in\Char(S)} \wedge^{\bullet}\left( \solvmfd; E_{\beta} \right)\to \bigoplus_{\beta\in{\mathcal A}^{\prime}} \wedge^{\bullet}\left( \solvmfd; E_{\beta} \right) \;. $$

Let $\beta\in{\mathcal A}^{\prime}$.
Take $\alpha\in  {\mathcal A}$ such that $E_{\beta}=E_{\alpha_{\ast}} $.
Then we have the inclusion
$$ \wedge^{\bullet}\g_{\C}^\ast \otimes V_{\alpha} \;\subseteq\; \wedge^{\bullet}\left( \solvmfd; E_{\beta} \right) \;, $$
and we consider the  F.~A. Belgun symmetrization map, \cite[Theorem 7]{belgun},
$$ \mu_{\alpha}\colon \wedge^{\bullet}\left( \solvmfd; E_{\beta} \right)\to \wedge^{\bullet}\g_{\C}^\ast\otimes V_{\alpha} \;. $$
We define the map 
$$ \Phi_{\beta} \;:=\; \sum_{\alpha\in {\mathcal A} \text{ s.t. }E_{\beta}=E_{\alpha_{\ast}}}\mu_{\alpha}: \wedge^{\bullet}\left( \solvmfd; E_{\beta} \right)\to \bigoplus_{\alpha\in {\mathcal A} \text{ s.t. } E_{\beta}=E_{\alpha_{\ast}}}\wedge^{\bullet}\g_{\C}^\ast\otimes V_{\alpha} \;.
$$

Then for distinct characters $\alpha$ and $\alpha^{\prime}$ with $E_{\alpha_{\ast}} = E_{\beta} = E_{\alpha^{\prime}_{\ast}}$, for the inclusion $\iota_{\alpha}\colon \wedge^{\bullet}\g^\ast_{\C}\otimes V_{\alpha}\to \wedge^{\bullet}\left( \solvmfd; E_{\beta} \right)$, we have  $\mu_{\alpha^{\prime}}\circ \iota_{\alpha} = 0$ (see the proof of \cite[Proposition 6.1]{kasuya-jdg}).
Hence, for  $\iota_{\beta}\colon \bigoplus_{\alpha\in {\mathcal A} \text{ s.t. } E_{\beta}=E_{\alpha_{\ast}}}\wedge^{\bullet}\g^\ast_{\C}\otimes V_{\alpha}\to \wedge^{\bullet}\left( \solvmfd; E_{\beta} \right)$, we have
$\Phi_{\beta}\circ i_{\beta}=\id$.
We define the map
$$\Phi \;:=\; \sum_{\beta\in {\mathcal A}^{\prime}}\Phi_{\beta} \colon \bigoplus_{\beta\in{\mathcal A}^{\prime}} \wedge^{\bullet}\left( \solvmfd; E_{\beta} \right)\to \bigoplus_{\alpha\in {\mathcal A}}\wedge^{\bullet}\g_{\C}^\ast\otimes V_{\alpha} \;.
$$
Then for $\iota \colon \bigoplus_{\alpha\in {\mathcal A}}\wedge^{\bullet}\g_{\C}^\ast\otimes V_{\alpha}\to \bigoplus_{\beta\in{\mathcal A}^{\prime}} \wedge^{\bullet}\left( \solvmfd; E_{\beta} \right)$, we have $\Phi\circ \iota=\id$.

Since $\omega$ is $G$-left-invariant, the map $\Phi$ commutes with the operators $L$ and $\Lambda$ and the operator $D^{\Lambda}$.
Since the $T$-action on $\bigoplus_{\alpha\in {\mathcal A}}\wedge^{\bullet} \g^\ast_\C\otimes V_{\alpha}$ is diagonalizable,
we can take  the direct sum
$$ \bigoplus_{\alpha\in {\mathcal A}}\wedge^{\bullet} \g^\ast_\C\otimes V_{\alpha} \;=\; \left(\bigoplus_{\alpha\in {\mathcal A}}\wedge^{\bullet} \g^\ast_\C\otimes V_{\alpha} \right)^{T}\oplus  D^{\bullet}
$$
of cochain complexes,
where $\left(\bigoplus_{\alpha\in {\mathcal A}}\wedge^{\bullet} \g^\ast_\C\otimes V_{\alpha} \right)^{T}$ is 
the sub-complex that consists of the elements of $\bigoplus_{\alpha\in {\mathcal A}}\wedge^{\bullet} \g\otimes V_{\alpha} $ fixed by the action of $T$ and $ D^{\bullet}$ is its complement for the action.
Hence the projection
$$ q \colon \bigoplus_{\alpha\in {\mathcal A}}\wedge^{\bullet} \g^\ast_\C\otimes V_{\alpha}\to \left(\bigoplus_{\alpha\in {\mathcal A}}\wedge^{\bullet} \g^\ast_\C\otimes V_{\alpha} \right)^{T}
$$
is a cochain complex map.
Since $\omega$ is $T$-invariant, the map $q \colon \bigoplus_{\alpha\in {\mathcal A}}\wedge^{\bullet} \g^\ast_\C\otimes V_{\alpha}\to \left(\bigoplus_{\alpha\in {\mathcal A}}\wedge^{\bullet} \g^\ast_\C\otimes V_{\alpha} \right)^{T}
$ commutes with the operators $L$ and $\Lambda$ and  the operator $D^{\Lambda}$.

Since we have
$$ \left(\bigoplus_{\alpha\in {\mathcal C}}\wedge^{\bullet} \g^{\ast}_{\C}\otimes V_{\alpha} \right)^{T} \;=\; \wedge^{\bullet}\langle x_{1}\otimes v_{\alpha_{1}},\dots,x_{n}\otimes v_{\alpha_{n}}\rangle \;,
$$
we have
$$
\left(\bigoplus_{\alpha\in {\mathcal C}}\wedge^{\bullet} \g^{\ast}_{\C}\otimes V_{\alpha} \right)^{T} \;=\; \left(\bigoplus_{\alpha\in {\mathcal A}}\wedge^{\bullet} \g^{\ast}_{\C}\otimes V_{\alpha} \right)^{T}
$$
and so we have
$$\wedge^{\bullet} \mathfrak{u}^{\ast} \;\simeq\; \left(\bigoplus_{\alpha\in {\mathcal A}}\wedge^{\bullet} \g^{\ast}_{\C}\otimes V_{\alpha} \right)^{T} \;.
$$

We consider the maps $\iota\colon\wedge^{\bullet} \mathfrak{u}^{\ast}\to \bigoplus_{\beta\in\Char(S)} \wedge^{\bullet}\left( \solvmfd; E_{\beta} \right)$ given by the composition
$$
\wedge^{\bullet} \mathfrak{u}^{\ast} \simeq \left(\bigoplus_{\alpha\in {\mathcal A}}\wedge^{\bullet} \g^{\ast}_{\C}\otimes V_{\alpha} \right)^{T} \to \bigoplus_{\alpha\in {\mathcal A}}\wedge^{\bullet} \g^{\ast}_{\C}\otimes V_{\alpha} \to \bigoplus_{\beta\in{\mathcal A}^{\prime}} \wedge^{\bullet}\left( \solvmfd; E_{\beta} \right) \to \bigoplus_{\beta\in\Char(S)} \wedge^{\bullet}\left( \solvmfd; E_{\beta} \right)
$$
and $\Psi\colon \bigoplus_{\beta\in\Char(S)} \wedge^{\bullet}\left( \solvmfd; E_{\beta} \right)\to \wedge^{\bullet} \mathfrak{u}^{\ast}$
given by the composition
$$ \bigoplus_{\beta\in\Char(S)} \wedge^{\bullet}\left( \solvmfd; E_{\beta} \right) \stackrel{p}{\to} \bigoplus_{\beta\in{\mathcal A}^{\prime}} \wedge^{\bullet}\left( \solvmfd; E_{\beta} \right) \stackrel{\Phi}{\to} \bigoplus_{\alpha\in {\mathcal A}}\wedge^{\bullet} \g^\ast_\C\otimes V_{\alpha} \stackrel{q}{\to} \left(\bigoplus_{\alpha\in {\mathcal A}}\wedge^{\bullet} \g\otimes V_{\alpha} \right)^{T} \simeq \wedge^{\bullet} \mathfrak{u}^{\ast} \;.
$$
Then by the above arguments, $\iota$ and $\Phi$ commute with the differentials $D$ and $D^{\Lambda}$ and satisfy $\Psi\circ \iota=\id$.

By Theorem \ref{min} the injection $\iota \colon \wedge^{\bullet} \mathfrak{u}^{\ast}\to \bigoplus_{\beta\in\Char(S)} \wedge^{\bullet}\left( \solvmfd; E_{\beta} \right)$ induces the isomorphism
$$H^{\bullet}_{D}(\mathfrak{u}) \stackrel{\simeq}{\to} H_{D}^{\bullet}\left( \bigoplus_{\beta\in\Char(S)} \wedge^{\bullet}\left( \solvmfd; E_{\beta} \right)\right) \;.
$$
By this isomorphism and the symplectic-$\star$-operator $\star_{\omega}$, the injection $\iota\colon \wedge^{\bullet} \mathfrak{u}^{\ast}\to \bigoplus_{\beta\in\Char(S)} \wedge^{\bullet}\left( \solvmfd; E_{\beta} \right)$ induces the isomorphism
$$H^{\bullet}_{D^{\Lambda}}(\mathfrak{u}) \stackrel{\simeq}{\to} H_{D^{\Lambda}}^{\bullet}\left( \bigoplus_{\beta\in\Char(S)} \wedge^{\bullet}\left( \solvmfd; E_{\beta} \right)\right) \;.
$$
Hence by using the map $\Psi \colon \bigoplus_{\beta\in\Char(S)} \wedge^{\bullet}\left( \solvmfd; E_{\beta} \right)\to \wedge^{\bullet} \mathfrak{u}^{\ast}$ as above,
the theorem follows from Corollary \ref{ives}.
\end{proof}

\begin{cor}\label{aep-min}
Let $G$ be a connected simply-connected solvable Lie group with a lattice $\Gamma$ and endowed with a $G$-left-invariant $T$-invariant symplectic structure $\omega$.

The above map $\iota\colon\wedge^{\bullet}\mathfrak{u}^{\ast}\to \bigoplus_{\beta\in\Char(S)} \wedge^{\bullet}\left( \solvmfd; E_{\beta} \right)$ induces the Aeppli cohomology isomorphism
$$ H^{\bullet}_{A}(\mathfrak{u}) \stackrel{\simeq}{\to} H^{\bullet}_{A}\left(\bigoplus_{\beta\in\Char(S)} \wedge^{\bullet}\left( \solvmfd; E_{\beta} \right)\right) \;. $$

\end{cor}
\begin{proof}
By the map $\Psi \colon \bigoplus_{\beta\in\Char(S)} \wedge^{\bullet}\left( \solvmfd; E_{\beta} \right)\to \wedge^{\bullet} \mathfrak{u}^{\ast}$ constructed in the proof of Theorem \ref{sym-min}, the induced map $$ H^{\bullet}_{A}(\mathfrak{u}) \to H^{\bullet}_{A}\left(\bigoplus_{\beta\in\Char(S)} \wedge^{\bullet}\left( \solvmfd; E_{\beta} \right)\right) \; $$
is injective.
Hence it is sufficient to show that there exists an isomorphism  $$ H^{\bullet}_{A}(\mathfrak{u}) \cong H^{\bullet}_{A}\left(\bigoplus_{\beta\in\Char(S)} \wedge^{\bullet}\left( \solvmfd; E_{\beta} \right)\right) \; .$$

Since $\mathfrak{u}$ is a nilpotent Lie algebra and $\omega$ can be regard as a symplectic form on $\mathfrak{u}$, like the duality between Bott-Chern and Aeppli cohomologies of compact symplectic manifolds, we have
 the isomorphism $ H^\bullet_{A}(\mathfrak{u}) \stackrel{\simeq}{\to} H^{2n-\bullet}_{BC}(\mathfrak{u})$ induced by Hodge-$*$-operator.
By Theorem \ref{sym-min}, we have an isomorphism $H^{2n-\bullet}_{BC}(\mathfrak{u})\cong H^{2n-\bullet}_{BC}\left(\bigoplus_{\beta\in\Char(S)} \wedge^{\bullet}\left( \solvmfd; E_{\beta} \right)\right) $.
Now for $\beta\in\Char(S)$, we have $\beta^{-1}\in\Char(S)$ and hence by the duality between Bott-Chern and Aeppli cohomologies of $\solvmfd$,  we have
\begin{multline*}
H^{2n-\bullet}_{BC}\left(\bigoplus_{\beta\in\Char(S)} \wedge^{\bullet}\left( \solvmfd; E_{\beta} \right)\right) \\
\cong H^\bullet_{A} \left(\bigoplus_{\beta\in\Char(S)} \wedge^{\bullet}\left( \solvmfd; E^{\ast}_{\beta} \right)\right)
=H^\bullet_{A} \left(\bigoplus_{\beta\in\Char(S)} \wedge^{\bullet}\left( \solvmfd; E_{\beta^{-1}} \right)\right)\\=H^\bullet_{A} \left(\bigoplus_{\beta\in\Char(S)} \wedge^{\bullet}\left( \solvmfd; E_{\beta} \right)\right).
\end{multline*}
Hence we have $$ H^{\bullet}_{A}(\mathfrak{u}) \cong H^{\bullet}_{A}\left(\bigoplus_{\beta\in\Char(S)} \wedge^{\bullet}\left( \solvmfd; E_{\beta} \right)\right) \; .$$
\end{proof}

We apply these results to untwisted symplectic cohomologies.
Consider $A_{\Gamma}^{\bullet}=\iota^{-1}\left(\wedge^{\bullet} \solvmfd\otimes \C\right)$ as in Note \ref{untwi-co}.
Then $A_{\Gamma}^{\bullet}$ is a sub-complex of the  bi-differential $\Z$-graded $\wedge^{\bullet} \solvmfd\otimes \C$ and  by Theorem \ref{sym-min} and Corollary \ref{aep-min}, the inclusion $A_{\Gamma}^{\bullet}\subset \wedge^{\bullet}\solvmfd \otimes \C$ induces isomorphisms
$$H^{\bullet}_{BC}(A_{\Gamma}^{\bullet}) \stackrel{\simeq}{\to} H^{\bullet}_{BC}\left(\solvmfd;\R\right)\otimes_\R\C $$and $$H^{\bullet}_{A}(A_{\Gamma}^{\bullet}) \stackrel{\simeq}{\to} H^{\bullet}_{A}\left(\solvmfd;\R\right)\otimes_\R\C.$$
Hence we have the result for symplectic Bott-Chern  and Aeppli cohomologies on general solvmanifolds as we give in Introduction.

\begin{thm}\label{thm:sympl-cohom-solvmfd}
Let $\solvmfd$ be a $2n$-dimensional solvmanifold endowed with a $G$-left-invariant symplectic structure $\omega$.
Consider the sub-complex $A_{\Gamma}^{\bullet}\subseteq \wedge^{\bullet} \solvmfd\otimes_\R \C$ defined in \eqref{eq:def-a1} (or \eqref{eq:def-a}) as in Note \ref{untwi-co}.
Suppose that $\omega\in A^{2}_{\Gamma}$.
Then, the inclusion $A_{\Gamma}^{\bullet} \hookrightarrow \wedge^{\bullet} \solvmfd\otimes_\R \C$  induces the isomorphisms
\[H^{\bullet}_{BC}(A_{\Gamma}^{\bullet}) \stackrel{\simeq}{\to} H^{\bullet}_{BC}\left(\solvmfd;\R\right)\otimes_\R\C \; \]
and
\[H^{\bullet}_{A}(A_{\Gamma}^{\bullet}) \stackrel{\simeq}{\to} H^{\bullet}_{A}\left(\solvmfd;\R\right)\otimes_\R\C \; \]
\end{thm}

\subsection{Examples}

As an explicit application of Theorem \ref{thm:sympl-cohom-solvmfd}, we compute the symplectic cohomologies of the (non-completely-solvable) Nakamura manifold.

\begin{ex}[The complex parallelizable Nakamura manifold]
Consider the Lie group
\[
G \;:=\; \C\ltimes_{\phi} \C^{2}
\qquad \text{ where } \qquad
\phi(z) \;:=\; \left(
\begin{array}{cc}
\esp^{z}& 0  \\
0&    \esp^{-z}  
\end{array}
\right) \;.
\]
There exist $a+\sqrt{-1}\,b \in \C$ and $c+\sqrt{-1}\,d\in \C$ such that $ \Z(a+\sqrt{-1}\,b)+\Z(c+\sqrt{-1}\,d)$ is a lattice in $\C$ and $\phi(a+\sqrt{-1}\,b)$ and $\phi(c+\sqrt{-1}\,d)$ are conjugate to elements of $\mathrm{SL}(4;\Z)$, where we regard $\mathrm{SL}(2;\C)\subset \mathrm{SL}(4;\R)$, see \cite{hasegawa-dga}.
Hence there exists a lattice $\Gamma := \left( \Z \left( a+\sqrt{-1}\,b \right) + \Z \left( c+\sqrt{-1}\,d \right) \right) \ltimes_{\phi} \Gamma^{\prime\prime}$ of $G$ such that $\Gamma^{\prime\prime}$ is a lattice of $\C^{2}$. Let $X := \left. \Gamma \middle\backslash G \right.$ be the \emph{complex parallelizable Nakamura manifold}, \cite[\S2]{nakamura}.

For a coordinate set $(z_{1},z_{2},z_{3})$ of $\C\ltimes_{\phi} \C^{2}$, we have the basis $\left\{ \frac{\partial}{\partial z_{1}}, \;  \esp^{z_{1}}\,\frac{\partial}{\partial z_{2}}, \; \esp^{-z_{1}}\,\frac{\partial}{\partial z_{3}} \right\}$ of the Lie algebra $\g_{+}$ of the $G$-left-invariant holomorphic vector fields on $G$ such that
$$ \left( \Ad_{(z_{1},z_{2},z_{3})} \right)_{\mathrm{s}} \;=\; \diag\left(1,\, \esp^{z_{1}},\, \esp^{-z_{1}}\right) \;\in\; \Aut(\g_+) \;. $$

\begin{enumerate}[{\itshape (a)}]
 \item\label{item:ex-nakamura2-1}
If $b \in \pi\,\Z$ and $d \in\pi\,\Z$, then, 
for $z\in \left(a+\sqrt{-1}\,b\right)\,\Z + \left(c+\sqrt{-1}\,d\right)\,\Z$, we have $\phi(z)\in \mathrm{SL}(2;\R)$.
In this case, the sub-complex $A_{\Gamma}^{\bullet} \subseteq \wedge^{\bullet} \solvmfd\otimes_\R \C$ defined in \eqref{eq:def-a1} is summarized in Table \ref{table:a-nak-cplx-paral-1}. (In order to shorten the notations, we write, for example, $\de z_{12\bar2}:=\de z_{1}\wedge\de z_{2}\wedge\de \bar z_{2}$.)

\begin{center}
\begin{table}[ht]
 \centering
\begin{tabular}{>{$\mathbf\bgroup}l<{\mathbf\egroup$} || >{$}l<{$}}
\toprule
\text{case {\itshape (\ref{item:ex-nakamura2-1})}} & A^{\bullet}_{\Gamma} \\
\toprule
0 & \C \left\langle 1 \right\rangle \\
\midrule[0.02em]
1 & \C \left\langle \de z_{1},\;\de  z_{\bar 1}   \right\rangle \\
\midrule[0.02em]
2 & \C \left\langle  \de z_{1\bar1},\; \de z_{23},\; \de z_{2\bar 3},\; \de z_{3\bar 2},\; \de z_{\bar2\bar3} \right\rangle \\
\midrule[0.02em]
3 & \C \left\langle \de z_{123},\; \de z_{12\bar3},\; \de z_{13\bar2},\; \de z_{3\bar1\bar2},\;\de z_{2\bar1\bar3},\;\de z_{\bar1\bar2\bar3},\;  \de z_{\bar123},\;  \de z_{1\bar2\bar3} \right\rangle \\
\midrule[0.02em]
4 & \C \left\langle \de z_{123\bar1},\;\de z_{13\bar1\bar2},\;\de z_{23\bar2\bar3},\; \de z_{12\bar1\bar3},\; \de z_{1\bar1\bar2\bar3} \right\rangle \\
\midrule[0.02em]
5 & \C \left\langle \de z_{23\bar1\bar2\bar 3} ,\; \de z_{123\bar2\bar3} \right\rangle \\
\midrule[0.02em]
6 & \C \left\langle \de z_{123\bar1\bar2\bar3} \right\rangle \\
\bottomrule
\end{tabular}
\caption{The cochain complex $A_{\Gamma}^{\bullet}$ in \eqref{eq:def-a1} for the complex parallelizable Nakamura manifold in case {\itshape (\ref{item:ex-nakamura2-1})}.}
\label{table:a-nak-cplx-paral-1}
\end{table}
\end{center}

 \item\label{item:ex-nakamura2-2}
If $b\not \in \pi\,\Z$ or $d\not \in\pi\,\Z$, then the sub-complex $A_{\Gamma}^{\bullet} \subseteq \wedge^{\bullet} \solvmfd\otimes_\R \C$ defined in \eqref{eq:def-a1} is given in Table \ref{table:a-nak-cplx-parall}.

\begin{center}
\begin{table}[ht]
 \centering
\begin{tabular}{>{$\mathbf\bgroup}l<{\mathbf\egroup$} || >{$}l<{$}}
\toprule
\text{case {\itshape (\ref{item:ex-nakamura2-2})}} & A^{\bullet}_{\Gamma} \\
\toprule
0 & \C \left\langle 1 \right\rangle \\
\midrule[0.02em]
1 & \C \left\langle \de z_{1},\;\de  z_{\bar 1}   \right\rangle \\
\midrule[0.02em]
2 & \C \left\langle  \de z_{1\bar1},\; \de z_{23},\;  \de z_{\bar2\bar3} \right\rangle \\
\midrule[0.02em]
3 & \C \left\langle \de z_{123},\;  \de z_{\bar1\bar2\bar3},\;  \de z_{\bar123},\;  \de z_{1\bar2\bar3} \right\rangle \\
\midrule[0.02em]
4 & \C \left\langle \de z_{123\bar1},\;\de z_{23\bar2\bar3},\;  \de z_{1\bar1\bar2\bar3} \right\rangle \\
\midrule[0.02em]
5 & \C \left\langle \de z_{23\bar1\bar2\bar 3} ,\; \de z_{123\bar2\bar3} \right\rangle \\
\midrule[0.02em]
6 & \C \left\langle \de z_{123\bar1\bar2\bar3} \right\rangle \\
\bottomrule
\end{tabular}
\caption{The cochain complex $A_{\Gamma}^{\bullet}$ in \eqref{eq:def-a1} for the complex parallelizable Nakamura manifold in case {\itshape (\ref{item:ex-nakamura2-2})}.}
\label{table:a-nak-cplx-parall}
\end{table}
\end{center}

\end{enumerate}

\medskip

Consider the $G$-left-invariant symplectic structure
\[\omega \;:=\; \de z_{1} \wedge \de \bar z_{1} + \de z_{2} \wedge \de z_{3} + \de \bar z_{2} \wedge \de \bar z_{3} \;. \]

Note that, in both case {\itshape (\ref{item:ex-nakamura2-1})} and case {\itshape(\ref{item:ex-nakamura2-2})}, we have $\omega\in A_{\Gamma}^{2}$.
The operator $\de$ on $A_{\Gamma}^{\bullet}$ is trivial and so also $\de^{\Lambda}$ is.
Hence we have the natural isomorphism $A_{\Gamma}^{\bullet} \stackrel{\simeq}{\to} H_{BC}^{\bullet}(\solvmfd)\otimes_\R\C$.
Since we have also the natural isomorphism $A_{\Gamma}^{\bullet} \stackrel{\simeq}{\to} H_{dR}^{\bullet}(\solvmfd;\C)$, the natural map $H_{BC}^{\bullet}(\solvmfd) \otimes_\R\C \to H_{dR}^{\bullet}(\solvmfd;\C)$ is an isomorphism.
Hence the $\de\de^{\Lambda}$-lemma holds, equivalently, the Hard Lefschetz Condition holds.

\begin{rem}[{\cite{kasuya-osaka}}]
 In particular, by the direct computations of Lefschetz operators, we can show that solvmanifolds $\solvmfd$ such that $G=\R^{n}\ltimes_{\phi} \R^{m}$ with a semi-simple action $\phi\colon \R^n \to \GL(\R^m)$ satisfy the Hard Lefschetz Condition, \cite[Corollary 1.5]{kasuya-osaka}.
 In particular, the completely-solvable Nakamura manifold satisfies the Hard Lefschetz Condition, \cite{kasuya-osaka}.
\end{rem}

\end{ex}

We investigate explicitly the Sawai manifold, \cite{Saw}, as an example of a symplectic solvmanifold satisfying the Hard Lefschetz Condition but not the twisted Hard Lefschetz Condition and not the $DD^\Lambda$-Lemma, see also \cite{kasuya-jdg}.
We compute the twisted symplectic Bott-Chern cohomology and the twisted minimal model.

\begin{ex}[The Sawai manifold {\cite{Saw}, \cite[Example 1]{kasuya-jdg}}] 
We consider the $8$-dimensional solvable Lie group
$$ G \;:=\; G_{1}\times \R $$
where $G_{1}$ is the matrix group defined as
$$
G_1 \;:=\; \left \{\left( \begin{array}{ccccccc}
\esp^{a_{1}t}&0&0&0&0&\esp^{-a_{3}t}\,x_{2}&y_{1}\\
0&\esp^{a_{2}t}&0&\esp^{-a_{1}t}\,x_{3}&0&0&y_{2}\\
0&0&\esp^{a_{3}t}&0&\esp^{-a_{2}t}\,x_{1}&0&y_{3}\\
0&0&0&\esp^{-a_{1}t}&0&0&x_{1}\\
0&0&0&0&\esp^{-a_{2}t}&0&x_{2}\\
0&0&0&0&0&\esp^{-a_{3}t}&x_{3}\\
0&0&0&0&0&0&1
\end{array}
\right) \st t, x_{1}, x_{2}, x_{3}, y_{1}, y_{2}, y_{3}\in\R\right\} \;,
$$
where $a_{1}$, $a_{2}$, $a_{3}$ are distinct real numbers such that $a_{1}+a_{2}+a_{3}=0$.

Let $\g$ be the Lie algebra of $G$ and $\g^{\ast}$ the dual of $\g$.
The cochain complex $\left(\wedge^\bullet \g^{\ast}, \, \de\right)$ is generated by a basis $\left\{\alpha ,\beta, \zeta_{1}, \zeta_{2}, \zeta_{3}, \eta _{1}, \eta _{2}, \eta _{3}\right\}$ of $\g^{\ast}$ such that
\begin{eqnarray*}
 \de\alpha &=& 0 \;, \\[5pt]
 \de\beta &=& 0 \;, \\[5pt]
 \de\zeta_{j} &=& a_{j}\,\alpha \wedge \zeta_{j} \qquad \text{ for } j\in\{1,2,3\} \;, \\[5pt]
 \de\eta_{1} &=& -a_{1}\, \alpha \wedge \eta_{1} - \zeta_{2} \wedge \zeta _{3} \;, \\[5pt]
 \de\eta_{2} &=& -a_{2}\, \alpha \wedge \eta_{2} - \zeta_{3} \wedge \zeta _{1} \;, \\[5pt]
 \de\eta_{3} &=& -a_{3}\, \alpha \wedge \eta_{3} - \zeta_{1} \wedge \zeta _{2} \;.
\end{eqnarray*}

In \cite[Theorem 1]{Saw}, H. Sawai showed that, for some $a_{1}, a_{2}, a_{3}\in \R$, the group $G$ has a lattice $\Gamma$ and $\solvmfd$ satisfies formality and has a $G$-invariant symplectic form,
\[ \omega \;:=\; \alpha \wedge \beta + \zeta _{1} \wedge \eta_{1} - 2\, \zeta _{2} \wedge \eta_{2} + \zeta_{3} \wedge \eta_{3} \;, \]
satisfying the Hard Lefschetz Condition.

We have
\[
\Ad_{s}(G) \;=\:
\left \{\left( \begin{array}{cccccccc}
\esp^{a_{1}t}&0&0&0&0&0&0&0\\
0&\esp^{a_{2}t}&0&0&0&0&0&0\\
0&0&\esp^{a_{3}t}&0&0&0&0&0\\
0&0&0&\esp^{-a_{1}t}&0&0&0&0\\
0&0&0&0&\esp^{-a_{2}t}&0&0&0\\
0&0&0&0&0&\esp^{-a_{3}t}&0&0\\
0&0&0&0&0&0&1&0\\
0&0&0&0&0&0&0&1
\end{array}
\right) \st t\in\R \right\} \;.
\]

Consider the $1$-dimensional representation
$$ \alpha_{1} \;:=\; \esp^{a_{1}t} \;. $$
Then, in \cite[Theorem 9.1]{kasuya-jdg}, the second author showed that $\left(\solvmfd,\, \omega\right)$ does not satisfy the $E_{\alpha_{1}}$-twisted Hard Lefschetz Condition.

We compute the two cohomologies $H^{\bullet}_{dR}\left(\solvmfd; E_{\alpha_{1}}\right)$ and $H^{\bullet}_{BC}\left(\solvmfd; E_{\alpha_{1}}\right)$ by using Theorem \ref{thm:twisted-sympl-bc-cohom-solvmfd}. The results of the computations are summarized in Table \ref{table:sawai-dr}, respectively Table \ref{table:sawai-bc}.

\begin{center}
\begin{table}[ht]
 \centering
\resizebox{\textwidth}{!}{
\begin{tabular}{>{$\mathbf\bgroup}l<{\mathbf\egroup$} || >{$}l<{$} | >{$}c<{$}}
\toprule
\text{$k$} & H^{k}_{dR}\left(\solvmfd; E_{\alpha_{1}}\right) & \dim_\C H^{k}_{dR}\left(\solvmfd; E_{\alpha_{1}}\right) \\
\toprule
\midrule[0.02em]
0 & 0 & 0 \\
\midrule[0.02em]
1 & \C \left\langle [\zeta_{1}]_{\alpha_{1}} \right\rangle & 1 \\
\midrule[0.02em]
2 & \C \left\langle  [\alpha\wedge\zeta_{1}]_{\alpha_{1}},\;  [\beta\wedge\zeta_{1}]_{\alpha_{1}}  \right\rangle & 2 \\
\midrule[0.02em]
3 & \C \left\langle  [\alpha\wedge\beta\wedge\zeta_{1}]_{\alpha_{1}},\; [\zeta_{1}\wedge \zeta_{2}\wedge\eta_{2}+\zeta_{1}\wedge \zeta_{3}\wedge\eta_{3}]_{\alpha_{1}} \right\rangle & 2 \\
\midrule[0.02em]
4 & \C \left\langle [\alpha\wedge\zeta_{1}\wedge \zeta_{2}\wedge\eta_{2}+\alpha\wedge\zeta_{1}\wedge \zeta_{3}\wedge\eta_{3}]_{\alpha_{1}}, [\beta\wedge\zeta_{1}\wedge \zeta_{2}\wedge\eta_{2}+\beta\wedge\zeta_{1}\wedge \zeta_{3}\wedge\eta_{3}]_{\alpha_{1}} \right\rangle & 2 \\
\midrule[0.02em]
5 & \C \left\langle [\alpha\wedge\beta\wedge\zeta_{1}\wedge \zeta_{2}\wedge\eta_{2}+\alpha\wedge\beta\wedge\zeta_{1}\wedge \zeta_{3}\wedge\eta_{3}]_{\alpha_{1}} \right\rangle & 1 \\
\midrule[0.02em]
6 & 0 & 0 \\
\midrule[0.02em]
7 & 0 & 0 \\
\midrule[0.02em]
8 & 0 & 0 \\
\bottomrule
\end{tabular}
}
\caption{$E_{\alpha_1}$-twisted de Rham cohomology $H^{\bullet}_{dR}\left(\solvmfd; E_{\alpha_{1}}\right)$ of the Sawai manifold.}
\label{table:sawai-dr}
\end{table}
\end{center}

\begin{center}
\begin{table}[ht]
 \centering
\resizebox{\textwidth}{!}{
\begin{tabular}{>{$\mathbf\bgroup}l<{\mathbf\egroup$} || >{$}l<{$} | >{$}c<{$}}
\toprule
\text{$k$} & H^{k}_{BC}\left(\solvmfd; E_{\alpha_{1}}\right) & \dim_\C H^{k}_{BC}\left(\solvmfd; E_{\alpha_{1}}\right) \\
\toprule
\midrule[0.02em]
0 & 0 & 0 \\
\midrule[0.02em]
1 & \C \left\langle [\zeta_{1}]_{\alpha_{1}} \right\rangle & 1 \\
\midrule[0.02em]
2 & \C \left\langle  [\alpha\wedge\zeta_{1}]_{\alpha_{1}},\;  [\beta\wedge\zeta_{1}]_{\alpha_{1}}  \right\rangle & 2 \\
\midrule[0.02em]
3 & \C \left\langle  [\alpha\wedge\beta\wedge\zeta_{1}]_{\alpha_{1}},\; [\zeta_{1}\wedge \zeta_{2}\wedge\eta_{2}]_{\alpha_{1}} ,\;[\zeta_{1}\wedge \zeta_{3}\wedge\eta_{3}]_{\alpha_{1}} \right\rangle & 3 \\
\midrule[0.02em]
4 & \C \left\langle [\alpha\wedge\zeta_{1}\wedge \zeta_{2}\wedge\eta_{2}]_{\alpha_{1}},\;[\alpha\wedge\zeta_{1}\wedge \zeta_{3}\wedge\eta_{3}]_{\alpha_{1}},\; [\beta\wedge\zeta_{1}\wedge \zeta_{2}\wedge\eta_{2}]_{\alpha_{1}},\;[\beta\wedge\zeta_{1}\wedge \zeta_{3}\wedge\eta_{3}]_{\alpha_{1}} \right\rangle & 4 \\
\midrule[0.02em]
5 & \C \left\langle [\zeta_{1}\wedge \zeta_{2}\wedge\eta_{2}\wedge \zeta_{3}\wedge\eta_{3}]_{\alpha_{1}},\;
 [\alpha\wedge\beta\wedge\zeta_{1}\wedge \zeta_{2}\wedge\eta_{2}]_{\alpha_{1}},\;[\alpha\wedge\beta\wedge\zeta_{1}\wedge \zeta_{3}\wedge\eta_{3}]_{\alpha_{1}} \right\rangle & 3 \\
\midrule[0.02em]
6 & [\alpha\wedge\zeta_{1}\wedge \zeta_{2}\wedge\eta_{2}\wedge \zeta_{3}\wedge\eta_{3}]_{\alpha_{1}},\;[\beta\wedge\zeta_{1}\wedge \zeta_{2}\wedge\eta_{2}\wedge \zeta_{3}\wedge\eta_{3}]_{\alpha_{1}} & 2 \\
\midrule[0.02em]
7 & [\alpha\wedge\beta\wedge\zeta_{1}\wedge \zeta_{2}\wedge\eta_{2}\wedge \zeta_{3}\wedge\eta_{3}]_{\alpha_{1}} & 1 \\
\midrule[0.02em]
8 & 0 & 0 \\
\bottomrule
\end{tabular}
}
\caption{$E_{\alpha_1}$-twisted Bott-Chern cohomology $H^{\bullet}_{BC}\left(\solvmfd; E_{\alpha_{1}}\right)$ of the Sawai manifold.}
\label{table:sawai-bc}
\end{table}
\end{center}

By these computations, the natural map $H^{\bullet}_{BC}\left(\solvmfd; E_{\alpha_{1}}\right) \to H^{\bullet}_{dR}\left(\solvmfd; E_{\alpha_{1}}\right)$ induced by the identity is surjective, and hence there is a symplectically-harmonic representative in each de Rham cohomology class with values in $E_{\alpha_{1}}$.
On the other hand, the natural map $H^{\bullet}_{BC}\left(\solvmfd; E_{\alpha_{1}}\right) \to H^{\bullet}_{dR}\left(\solvmfd; E_{\alpha_{1}}\right)$ induced by the identity is not injective, and hence the $D D^{\Lambda}$-Lemma does not hold.

Next we consider the twisted minimal model $\wedge^{\bullet}\mathfrak{u}$.
Consider
$$ \alpha_{2} \;:=\; \esp^{a_{2}t} \;, \qquad \alpha_{3} \;:=\; \esp^{a_{3}t} \;, $$
and define, for $j\in\{1,2,3\}$,
$$ \breve\alpha \;:=\; \alpha \;, \qquad \breve\beta \;:=\; \beta \;, \qquad \breve \zeta_{j} \;:=\; \zeta_{j}\otimes v_{\alpha_{j}} \;, \qquad \text{ and } \qquad \breve \eta_{j} \;:=\; \eta_{j}\otimes v_{\alpha^{-1}_{j}} \;. $$
Then we have
$$ \wedge^{\bullet}\mathfrak{u} \;=\; \wedge^{\bullet} \left\langle \breve\alpha,\, \breve\beta,\, \breve \zeta_{1},\, \breve \zeta_{2},\, \breve \zeta_{3},\, \breve \eta_{1},\, \breve \eta_{2},\, \breve \eta_{3} \right\rangle $$
such that
\begin{eqnarray*}
\de\breve\alpha \;=\; \de\breve\beta \;=\; \de\breve\zeta_{1} \;=\; \de\breve\zeta_{2} \;=\; \de\breve\zeta_{3} &=& 0 \;, \\[5pt]
\de\breve\eta_{1} &=& \breve\zeta_{2}\wedge\breve\zeta_{3} \;, \\[5pt]
\de\breve\eta_{2} &=& \breve \zeta_{3}\wedge\breve\zeta_{1} \;, \\[5pt]
\de\breve\eta_{3} &=& \breve\zeta_{1}\wedge \breve\zeta_{2} \;.
\end{eqnarray*}

We have 
$$ \omega \;=\; \breve\alpha\wedge \breve\beta +\breve\zeta _{1}\wedge \breve\eta_{1}-2\,\breve\zeta_{2}\wedge \breve\eta_{2} +\breve\zeta_{3}\wedge \breve\eta_{3} \;. $$

We have
$$ H^{\bullet}_{BC}({\mathfrak{u}}) \;\simeq\; H^{\bullet}_{BC}\left(\bigoplus_{\beta\in\Char(S)} \wedge^{\bullet}\left( \solvmfd; E_{\beta} \right)\right) \;.
$$
\end{ex}


\begin{thebibliography}{10}
\bibitem{aeppli}
A. Aeppli, On the cohomology structure of Stein manifolds, {\em Proc. Conf. Complex Analysis (Minneapolis, Minn., 1964)}, Springer, Berlin, 1965, pp. 58--70.

\bibitem{alessandrini-bassanelli}
L. Alessandrini, G. Bassanelli, Small deformations of a class of compact non-K\"ahler manifolds, {\em Proc. Amer. Math. Soc.} \textbf{109} (1990), no.~4, 1059--1062.

\bibitem{angella-1}
D. Angella, The cohomologies of the Iwasawa manifold and of its small deformations, {\em J. Geom. Anal.} \textbf{23} (2013), no.~3, 1355-1378.

\bibitem{angella-franzini-rossi}
D. Angella, M.~G. Franzini, F.~A. Rossi, Degree of non-K\"ahlerianity for $6$-dimensional nilmanifolds, {\em Manuscr. Math.} \textbf{148} (2015), no.~1--2, 177--211.

\bibitem{angella-kasuya-1}
D. Angella, H. Kasuya, Bott-Chern cohomology of solvmanifolds, \arxiv{1212.5708v3 [math.DG]}.

\bibitem{angella-tomassini-3}
D. Angella, A. Tomassini, On the $\partial\overline\partial$-Lemma and Bott-Chern cohomology, {\em Invent. Math.} \textbf{192} (2013), no.~1, 71--81.

\bibitem{angella-tomassini-4}
D. Angella, A. Tomassini, Symplectic manifolds and cohomological decomposition, {\em J. Symplectic Geom.} \textbf{12} (2014), no.~2, 215--236.

\bibitem{angella-tomassini-5}
D. Angella, A. Tomassini, Inequalities {\itshape à la} Fr\"olicher and cohomological decompositions, {\em J. Noncommut. Geom.} \textbf{9} (2015), no.~2, 505--542.

\bibitem{belgun}
F.~A. Belgun, On the metric structure of non-K\"ahler complex surfaces, {\em Math. Ann.} \textbf{317} (2000), no.~1, 1--40.

\bibitem{benson-gordon-nilmanifolds}
Ch. Benson, C.~S. Gordon, K\"ahler and symplectic structures on nilmanifolds, {\em Topology} \textbf{27} (1988), no.~4, 513--518.

\bibitem{bismut}
J.-M. Bismut, Hypoelliptic Laplacian and Bott-Chern cohomology, preprint (Orsay) (2011).

\bibitem{bock}
Ch. Bock, On Low-Dimensional Solvmanifolds, \arxiv{arXiv:0903.2926v4 [math.DG]}.

\bibitem{bott-chern}
R. Bott, S.~S. Chern, Hermitian vector bundles and the equidistribution of the zeroes of their holomorphic sections, {\em Acta Math.} \textbf{114} (1965), no.~1, 71--112.

\bibitem{brylinski}
J.-L. Brylinski, A differential complex for Poisson manifolds, {\em J. Differ. Geom.} \textbf{28} (1988), no.~1, 93--114.

\bibitem{cannasdasilva}
A. Cannas da Silva, {\em Lectures on symplectic geometry}, Lecture Notes in Mathematics, \textbf{1764}, Springer-Verlag, Berlin, 2001. pp. 

\bibitem{cavalcanti}
G.~R. Cavalcanti, New aspects of the $dd^c$-lemma, Oxford University D. Phil thesis, \texttt{arXiv:math/0501406v1 [math.DG]}.

\bibitem{cavalcanti-impa}
G.~R. Cavalcanti, {\em Introduction to generalized complex geometry}, Publicações Matemáticas do IMPA, 26$^\text{o}$ Colóquio Brasileiro de Matemática, Instituto Nacional de Matemática Pura e Aplicada (IMPA), Rio de Janeiro, 2007.

\bibitem{cavalcanti-jgp}
G.~R. Cavalcanti, The decomposition of forms and cohomology of generalized complex manifolds, {\em J. Geom. Phys.} \textbf{57} (2006), no.~1, 121--132.

\bibitem{cavalcanti-gualtieri}
G.~R. Cavalcanti, M. Gualtieri, Generalized complex structures on nilmanifolds, {\em J. Symplectic Geom.} \textbf{2} (2004), no.~3, 393--410.

\bibitem{ceballos-otal-ugarte-villacampa}
M. Ceballos, A. Otal, L. Ugarte, R. Villacampa, Invariant complex structures on $6$-nilmanifolds: classification, Fr\"olicher spectral sequence and special Hermitian metrics, {\em J. Geom. Anal.} (2014), DOI: 10.1007/s12220-014-9548-4.

\bibitem{connes}
A. Connes, Noncommutative differential geometry, {\em Inst. Hautes Études Sci. Publ. Math.} \textbf{62} (1985), 257--360.

\bibitem{cordero-fernandez-ugarte-gray}
L.~A. Cordero, M. Fern\'{a}ndez, L. Ugarte, A. Gray, A general description in the terms of the Fr\"olicher spectral sequence, {\em Differ. Geom. Appl.} \textbf{7} (1997), no.~1, 75--84.

\bibitem{deligne-griffiths-morgan-sullivan}
P. Deligne, Ph. Griffiths, J. Morgan, D.~P. Sullivan, Real homotopy theory of K\"ahler manifolds, {\em Invent. Math.} \textbf{29} (1975), no.~3, 245--274.

\bibitem{DER}
N. Dungey, A.~F.~M. ter Elst, D.~W. Robinson, {\em Analysis on Lie Groups with Polynomial Growth}, Progress in Mathematics, \textbf{214}, Birkh\"auser Boston (2003).

\bibitem{fernandez-Munoz-Ugarte}
M. Fernández, V. Muñoz, L. Ugarte, The $d\delta$-lemma for weakly Lefschetz symplectic manifolds, {\em Differential geometry and its applications}, 229--246, Matfyzpress, Prague, 2005.

\bibitem{griffiths-harris}
Ph. Griffiths, J. Harris, {\em Principles of algebraic geometry}, Reprint of the 1978 original, Wiley Classics Library, John Wiley \& Sons, Inc., New York, 1994.

\bibitem{goodwillie}
T.~G. Goodwillie, Cyclic homology, derivations, and the free loopspace, {\em Topology} \textbf{24} (1985), no.~2, 187--215.

\bibitem{goze-khakimdjanov}
M. Goze, Y. Khakimdjanov, {\em Nilpotent Lie algebras}, Mathematics and its Applications, \textbf{361}, Kluwer Academic Publishers Group, Dordrecht, 1996.

\bibitem{gualtieri}
M. Gualtieri, Generalized complex geometry, {\em Ann. of Math. (2)} \textbf{174} (2011), no.~1, 75--123.

\bibitem{guillemin}
V. Guillemin, Symplectic Hodge theory and the $\de\delta$-Lemma, preprint, Massachusetts Insitute of Technology, 2001.

\bibitem{Hai}
R.~M. Hain, The Hodge de Rham theory of relative Malcev completion, {\em Ann. Sci. Ecole Norm. Sup. (4)} \textbf{31} (1998), no.~1, 47--92.

\bibitem{hasegawa}
K. Hasegawa, Minimal models of nilmanifolds, {\em Proc. Amer. Math. Soc.} \textbf{106} (1989), no.~1, 65--71.

\bibitem{hasegawa-dga}
K. Hasegawa, Small deformations and non-left-invariant complex structures on six-dimensional compact solvmanifolds, {\em Differ. Geom. Appl.} {\bf 28} (2010), no.~2, 220--227.

\bibitem{hattori}
A. Hattori, Spectral sequence in the de {R}ham cohomology of fibre bundles, {\em J. Fac. Sci. Univ. Tokyo Sect. I} \textbf{8} (1960), no.~1960, 289--331.

\bibitem{kasuya-osaka}
H. Kasuya, Formality and hard Lefschetz properties of aspherical manifolds, \textit{ Osaka J. Math.} {\bf 50} (2013) no. 2.

\bibitem{kasuya-jdg}
H. Kasuya, Minimal models, formality and hard Lefschetz properties of solvmanifolds with local systems,  {\em J. Differ. Geom.} \textbf{93}, (2013), 269--298.

\bibitem{kooistra}
R. Kooistra, Regulator currents on compact complex manifolds, Thesis (Ph.D.)–University of Alberta (Canada), 2011.

\bibitem{koszul}
J.-L. Koszul, Crochet de Schouten-Nijenhuis et cohomologie, The mathematical heritage of Élie Cartan (Lyon, 1984), {\em Astérisque} \textbf{1985}, Numero Hors Serie, 257--271.

\bibitem{latorre-ugarte-villacampa}
A. Latorre, L. Ugarte, R. Villacampa, On the Bott-Chern cohomology and balanced Hermitian nilmanifolds, {\em Internat. J. Math.} \textbf{25} (2014), no.~6, 1450057, 24 pp..

\bibitem{macri}
M. Macrì, Cohomological properties of unimodular six dimensional solvable Lie algebras, \arxiv{arXiv:1111.5958v2 [math.DG]}.

\bibitem{magnin}
L. Magnin, Sur les algèbres de Lie nilpotentes de dimension $\leq 7$, {\em J. Geom. Phys.} \textbf{3} (1986), no.~1, 119--144.

\bibitem{mathieu}
O. Mathieu, Harmonic cohomology classes of symplectic manifolds, {\em Comment. Math. Helv.} \textbf{70} (1995), no.~1, 1--9.

\bibitem{merkulov}
S.~A. Merkulov, Formality of canonical symplectic complexes and Frobenius manifolds, {\em Int. Math. Res. Not.} \textbf{1998} (1998), no.~14, 727--733.

\bibitem{mccleary}
J. McCleary, {\em A user's guide to spectral sequences}, Second edition, Cambridge Studies in Advanced Mathematics, \textbf{58}, Cambridge University Press, Cambridge, 2001.

\bibitem{milnor}		
J. Milnor, Curvature of left-invariant metrics on Lie groups, {\em Adv. Math.} {\bf 21} (1976), no.~3, 293--329.

\bibitem{morozov}
V.~V. Morozov, Classification of nilpotent Lie algebras of sixth order, {\em Izv. Vys\v{s}. U\v{c}ebn. Zaved. Matematika} \textbf{1958} (1958), no.~4 (5), 161--171.

\bibitem{mostow}
G.~D. Mostow, Cohomology of topological groups and solvmanifolds, {\em Annals of Math. (2)} \textbf{73} (1961), no.~1, 20--48.

\bibitem{nakamura}
I. Nakamura, Complex parallelisable manifolds and their small deformations, {\em J. Differ. Geom.} \textbf{10} (1975), no.~1, 85--112.

\bibitem{sage}
\emph{{S}age {M}athematics {S}oftware ({V}ersion 6.10)}, The Sage Developers, 2015, {\url{http://www.sagemath.org}}.

\bibitem{salamon}
S.~M. Salamon, Complex structures on nilpotent Lie algebras, {\em J. Pure Appl. Algebra} \textbf{157} (2001), no.~2-3, 311--333.

\bibitem{Saw}
H. Sawai, A construction of lattices on certain solvable Lie groups, {\em Topology Appl.} \textbf{154} (2007), no.~18, 3125--3134.

\bibitem{schweitzer}
M. Schweitzer, Autour de la cohomologie de Bott-Chern, Prépublication de l’Institut Fourier no. 703 (2007), \arxiv{arXiv:0709.3528v1 [math.AG]}.

\bibitem{Sim}
C.~T. Simpson, Higgs bundles and local systems, {\em Inst. Hautes \'{E}tudes Sci. Publ. Math. No.} \textbf{75} (1992), 5--95.

\bibitem{tardini-tomassini}
N. Tardini, A. Tomassini, On the cohomology of almost complex and symplectic manifolds and proper surjective maps, \texttt{arXiv:1601.08146}.

\bibitem{tseng-yau-1}
L.-S. Tseng, S.-T. Yau, Cohomology and Hodge Theory on Symplectic Manifolds: I, {\em J. Differ. Geom.} \textbf{91} (2012), no.~3, 383--416.

\bibitem{tseng-yau-2}
L.-S. Tseng, S.-T. Yau, Cohomology and Hodge Theory on Symplectic Manifolds: II, {\em J. Differ. Geom.} \textbf{91} (2012), no.~3, 417--443.

\bibitem{tseng-yau-3}
L.-S. Tseng, S.-T. Yau, Generalized Cohomologies and Supersymmetry, {\em Comm. Math. Phys.} \textbf{326} (2014), no.~3, 875--885.

\bibitem{varouchas}
J. Varouchas, Sur l'image d'une vari\'{e}t\'{e} K\"ahl\'{e}rienne compacte, in F. Norguet (ed.), {\em Fonctions de plusieurs variables complexes V}, S\'{e}min. F. Norguet, Paris 1979-1985, Lect. Notes Math. {\bfseries 1188}, 245--259 (1986).

\bibitem{warner}
F.~W. Warner, {\em Foundations of differentiable manifolds and Lie groups}, Corrected reprint of the 1971 edition, Graduate Texts in Mathematics, \textbf{94}, Springer-Verlag, New York-Berlin, 1983.

\bibitem{yan}
D. Yan, Hodge structure on symplectic manifolds, {\em Adv. Math.} \textbf{120} (1996), no.~1, 143--154.

\end{thebibliography}
\end{document}